\documentclass[10pt,namelimits,sumlimits]{amsart}
\usepackage{appendix}
\usepackage{amssymb,amsmath, mathrsfs}
\usepackage[mathscr]{eucal}
\usepackage{color}
\usepackage{enumitem}
\usepackage[utf8]{inputenc}
\usepackage[leftcaption]{sidecap}
\usepackage{tikz}
 \usepackage{tikz-cd}
 \usepackage{stmaryrd}
\usetikzlibrary{matrix,arrows,decorations.pathmorphing}
\usetikzlibrary{positioning}
\usetikzlibrary{matrix,arrows,decorations.pathmorphing, shapes.geometric}

\usepackage{psfrag}

\usepackage{epstopdf}
\usepackage{graphicx}
\usepackage{verbatim}
\usepackage{amsmath,amsthm}
\usepackage{graphics}
\usepackage{color}
\usepackage{epsfig}
\usepackage{amssymb,amsmath}
\usepackage[mathscr]{eucal}
\usepackage{ upgreek }
\usepackage{tensor}
\usepackage{physics}

\DeclareFontFamily{U}{mathx}{\hyphenchar\font45}
\DeclareFontShape{U}{mathx}{m}{n}{
      <5> <6> <7> <8> <9> <10>
      <10.95> <12> <14.4> <17.28> <20.74> <24.88>
      mathx10
      }{}
\DeclareSymbolFont{mathx}{U}{mathx}{m}{n}
\DeclareFontSubstitution{U}{mathx}{m}{n}
\DeclareMathAccent{\widecheck}{0}{mathx}{"71}

\newtheorem{theorem}[equation]{Theorem}
\newtheorem{question}[equation]{Question}

\newtheorem{lemma}[equation]{Lemma}
\newtheorem{prop}[equation]{Proposition}
\newtheorem{cor}[equation]{Corollary}

\newtheorem{definition}[equation]{Definition}

\theoremstyle{remark}
\newtheorem{remark}[equation]{Remark}
\newtheorem{notation}[equation]{Notation}
\newtheorem{convention}[equation]{Convention}
\newtheorem{assumption}[equation]{Assumption}

\numberwithin{equation}{section}

\newcommand{\Sim}{\displaystyle\operatornamewithlimits{\sim}}
\newcommand{\Stab}{\mathrm{Stab}}
\newcommand{\Span}{\mathrm{Span}}

\newcommand{\C}{\mathbb{C}}
\newcommand{\R}{\mathbb{R}}

\newcommand{\Z}{\mathbb{Z}}

\newcommand{\N}{\mathbb{N}}

\newcommand{\T}{\mathbb{T}}
\newcommand{\cat}{{\widecheck{K}}}

\newcommand{\Lcal}{\mathcal{L}}
\newcommand{\Ecal}{\mathcal{E}}

\newcommand{\Fcal}{\mathcal{F}}
\newcommand{\Pcal}{\mathcal{P}}
\newcommand{\sss}{{{\ensuremath{\mathrm{s}}}}}
\newcommand{\ssstilde}{\widetilde{\sss}}
\newcommand{\zz}{\ensuremath{{\mathrm{z}}}}

\newcommand{\group}{{\mathscr{G}  }}
\newcommand{\groupL}{{\group^L_{\sym}}}

\newcommand{\Phat}{\widehat{\Phi}}
\newcommand{\Phip}{\Phi'}
\newcommand{\zetabold}{{\boldsymbol{\zeta}}}
\newcommand{\zerobold}{{\boldsymbol{0}}}
\newcommand{\taubold}{{\boldsymbol{\tau}}}

\newcommand{\gbold}{\bold{g}}

\newcommand{\cunder}{\underline{c}}

\newcommand{\sym}{\mathrm{sym}}

\newcommand{\Psibold}{{\boldsymbol{\Psi}}}

\newcommand{\Sph}{\mathbb{S}}

\newcommand{\Ghat}{\widehat{G}}
\newcommand{\Ccal}{\mathcal{C}}

\newcommand{\id}{\operatorname{Id}}

\newcommand{\gtilde}{\widetilde{g}}

\newcommand{\Lcaltilde}{\widetilde{\Lcal}}

\newcommand{\Mcal}{\mathcal{M}}

\newcommand{\Rcap}{\mathsf{R}}
\newcommand{\Rcapunder}{\underline{\Rcap}}
\newcommand{\Tcap}{\mathsf{T}}

\newcommand{\dbold}{{\mathbf{d}}}

\newcommand{\vunder}{\underline{v}}
\newcommand{\osc}{{{\mathrm{osc}}}}
\newcommand{\ave}{{{\mathrm{avg}}}}
\newcommand{\domzb}{B_{\Pcal}}

\newcommand{\graph}{\operatorname{Graph}}

\newcommand{\tauunder}{\underline{\tau}}

\newcommand{\skernel}{\mathscr{K}}
\newcommand{\skernelv}{\widehat{\mathscr{K}}}
\newcommand{\val}{\mathscr{V}}

\newcommand{\Mbreve}{\breve{M}}
\newcommand{\Sigmabreve}{\breve{\Sigma}}
\newcommand{\Mhat}{\breve{M}}

\newcommand{\Ecalunder}{\underline{{\mathcal{E}}}} 
\newcommand{\kappaunder}{{\underline{\kappa}}}
\newcommand{\kappaunderbold}{{\boldsymbol{\underline{\kappa}}}}
\newcommand{\kappaperp}{\kappa^{\perp}}
\newcommand{\gammagl}{{\alpha}}
\newcommand{\disjun}{\textstyle\bigsqcup}
\newcommand{\zetaboldhatunder}{{\breve{\zetaboldunder}}}
\newcommand{\zetaboldunder}{\underline{\zetabold}}
\newcommand{\zetaboldhat}{{\breve{\boldsymbol{\zeta}}}}

\newcommand{\upphihat}{{\breve{\upphi}}}
\newcommand{\Omegahat}{\breve{\Sigma}}
\newcommand{\ubreve}{\breve{u}}
\newcommand{\ucheck}{\widecheck{u}}
\newcommand{\Dbreve}{\breve{D}}
\newcommand{\avg}{\operatornamewithlimits{avg}}
\newcommand{\Lapone}{\mathcal{E}_{\lambda_1}}
\newcommand{\Coord}{\mathcal{C}}
\newcommand{\low}{\mathrm{low}}
\newcommand{\high}{\mathrm{high}}

\newcommand{\wlow}{w_{\low}}
\newcommand{\whigh}{w_{\high}}
\newcommand{\wave}{w_{\ave}}
\newcommand{\ucir}{\mathring{u}}
\newcommand{\wcir}{\mathring{w}}
\newcommand{\gcir}{\mathring{g}}

\newcommand{\rtop}{{\overline{r}}}
\newcommand{\runder}{{\underline{r}}}
\newcommand{\gpeuc}{\mathring{g}}
\newcommand{\gcheck}{\widecheck{g}}

\newcommand{\Omegap}{\Omega_p}
\newcommand{\Xpm}{X_{\pm}}
\newcommand{\Xp}{X_+}
\newcommand{\Xm}{X_-}

\newcommand{\Eappr}{\Ecal_{\mathrm{low}}}
\newcommand{\muboldtilde}{\widetilde{\boldsymbol{\mu}}}
\newcommand{\mutilde}{\widetilde{\mu}}
\newcommand{\Spheq}{\mathbb{S}^2_{eq}}
\newcommand{\wosc}{w_{\osc}}
\newcommand{\Lpar}{L_{\mathrm{par}}}
\newcommand{\Omegapar}{\Omega_{\mathrm{par}}}
\begin{document}

\title[New minimal surface doublings]{New minimal surface doublings of the Clifford torus \\ and contributions to questions of Yau}

\author[N.~Kapouleas]{Nikolaos~Kapouleas}
\author[P.~McGrath]{Peter~McGrath}

\address{Department of Mathematics, Brown University, Providence, RI 02912}  
\email{nicolaos\_kapouleas@brown.edu}

\address{Department of Mathematics, North Carolina State University, Raleigh, NC 27695}
\email{pjmcgrat@ncsu.edu}

\date{\today}

\begin{abstract} 
The purpose of this article is three-fold.  
First, we apply a general theorem from our earlier work to produce many new minimal doublings of the Clifford Torus in the round three-sphere.  
This construction generalizes and unifies prior doubling constructions for the Clifford Torus, producing doublings with catenoidal bridges arranged along parallel copies of torus knots.  
Ketover has also constructed similar minimal surfaces by min-max methods as suggested by Pitts-Rubinstein, 
but his methods apply only to surfaces which are lifts of genus two surfaces in lens spaces, while ours are not constrained this way. 

Second, we use this family to prove a new, quadratic lower bound for the number of embedded minimal surfaces in $\mathbb{S}^3$ with prescribed genus.  
This improves upon bounds recently given by Ketover and Karpukhin-Kusner-McGrath-Stern, 
and contributes to a question of Yau about the structure of the space of minimal surfaces in $\mathbb{S}^3$ with fixed genus. 

Third, we verify Yau's conjecture for the first eigenvalue of minimal surfaces in $\mathbb{S}^3$ in the following cases. 
First, 
for all minimal surface doublings of the equatorial two-sphere constructible by our earlier general theorem. 
Second, for all the Clifford Torus doublings constructed in this article.  
\end{abstract}
\maketitle

\section{Introduction}
\label{Sintro}

\subsection*{The general framework}
\phantom{ab}

The first purpose of this article is to generalize and unify prior Clifford torus doubling constructions \cite{PittsR, KY, Wiygul, LDG, Ketover}. 
We proceed now to briefly review these earlier constructions. 

Most of the known closed embedded minimal surfaces in the round three-sphere $\Sph^3$ 
are known or expected to be either \emph{desingularizations} or \emph{doublings} of great two-spheres and Clifford tori,    
where we use the terms \emph{desingularizations} in the sense of \cite[Definition 1.3]{KapSurvey} and the further discussion there, 
and \emph{doublings} in the sense of \cite[Definition 1.1]{LDG}. 
For example the Lawson surfaces $\xi_{k,m}$ are desingularizations of $k+1$ great two-spheres intersecting symmetrically along a common great circle for $k\ge1$, $m\ge2$,  
and five of the nine minimal surfaces constructed in \cite{KPS} are doublings of $\Spheq$ in $\Sph^3(1)$.

In 1988 Pitts-Rubinstein proposed the construction by min-max methods of several new families of closed embedded minimal surfaces in $\Sph^3$ some of which 
are expected to be Clifford Torus doublings (examples 11-14 in Table 1 of \cite{PittsR}).  
In 2010, using PDE gluing methods, Kapouleas-Yang \cite{KY} constructed the first such examples.  
The idea of doubling constructions by PDE gluing methods is to construct approximately minimal surfaces by joining two copies of the \emph{base surface} ($\T$ in this case) with 
small catenoidal bridges and then correct to minimality by solving a PDE. 
The family constructed in \cite{KY} is consistent with example 11 from \cite{PittsR} and is parametrized by a number $m \in \N$ which needs to be assumed large enough 
because of the gluing methodology.  
For each such $m $ the minimal surface constructed satisfies the symmetries $\group^L_{\sym}$ of an $m\times m$ square lattice $L \subset \T \subset \Sph^3$ 
and contains one catenoidal bridge at each point of the lattice, so its genus is $m^2 + 1$. 

In 2020 Ketover-Marques-Neves produced similar surfaces by min-max methods in accordance to Pitts-Rubinstein's suggestion in \cite[Example 11]{PittsR}.  
Note that their surfaces are strongly expected to be identical to the Kapouleas-Yang surfaces but this has not yet been proved. 
In the constructions we discuss below it is often the case that different constructions provide surfaces with similar features which are strongly expected to be identical,  
and it is common to refer to them in the literature as being the same, 
although proving this remains a challenging open problem \cite[Remark 5.23]{LDG}. 

Wiygul \cite{Wiygul} generalized the construction in \cite{KY} to the case of $km\times lm$ rectangular lattices $L$, where $m$ is large in terms of given arbitrarily $k, l\geq 1$, 
and also to minimal \emph{stackings} of $\T$ which resemble more than two copies of $\T$ joined by catenoidal bridges.

In 2017 NK introduced a refinement of the original PDE gluing methodology for doubling constructions which he called \emph{Linearized Doubling (LD)}. 
Based on the LD methodology the authors proved a general theorem \cite[Theorem A]{LDG} which reduces the construction of a minimal surface doubling 
to the construction of suitable families of \emph{Linearized Doubling (LD) solutions}, 
that is singular solutions of the Jacobi equation with logarithmic singularities. 
As an application of this theorem they constructed \cite[Theorem 6.18]{LDG} doublings based on rectangular $k\times m$ lattices with any $k\in\N\setminus\{1,2\}$ and $m$ large enough in terms of $k$, 
with one bridge on each $\group^L_{\sym}$-fundamental domain, and also \cite[Theorem 6.25]{LDG} doublings with more than one bridge on each $\group^L_{\sym}$-fundamental domain.  
The first type of examples includes the family described in item 14 of Table 1 from \cite{PittsR}.

Using min-max theory, Ketover \cite{Ketover} constructed Clifford Torus doublings which are lifts of genus-two surfaces in lens space quotients.    
In more detail, for each lens space $L(p, q)$ with $q \notin \{1, p-1\}$, a one-parameter sweepout arising from a foliation of $L(p, q)$ constant mean curvature tori 
gives rise to a genus-two minimal surface in $L(p, q)$ which lifts to a minimal surface in $\Sph^3$.  
For a given $(a,b)$-torus knot and positive integer $k$, Ketover shows \cite[Lemma 5.4]{Ketover} that for appropriate sequences $(p_i, q_i)$, 
the lifts of the corresponding surfaces in $L(p_i, q_i)$ converge as varifolds to a multiplicity-two Clifford torus, 
with genus concentrating along $k$ parallel copies of the $(a,b)$-torus knot.  
Because the associated group-actions on $\Sph^3$ are free, the number $m$ of bridges on each such knot is constrained in terms of $(a, b)$ and $k$, 
and cannot be prescribed arbitrarily, even when large.  
The restrictions on $m$ are described in detail in \cite[Lemma 5.4]{Ketover}.

In this article we generalize our earlier constructions in \cite[Theorem 6.18]{LDG} by using torus knots to position the catenoidal bridges in the fashion of 
\cite[Example 13]{PittsR} and \cite{Ketover} but without any restrictions due to the difficulties associated with the min-max approach. 
This allows us to produce many more examples as described below. 
Moreover the geometry of these examples is precisely controlled by the gluing construction and this allows us to answer questions which cannot be addressed by min-max constructions. 

This way we contribute to two questions of Yau included in his extremely influential and inspirational collection of over a hundred problems and questions in Differential Geometry 
which he presented in 1982 \cite{Yau:problems}.   
These problems have been a source of inspiration which has driven progress in the field for over 40 years, 
and work towards resolving them has led to tremendous advances as for example in \cite{Song, BrendleLawson, MarquesNeves}.  
The first question is as follows.

\begin{question}[Yau (1982), {\cite[Problem 96]{Yau:problems}}]
What is the structure of the space of minimal surfaces of a fixed genus in $\Sph^3$?  
\end{question}

Let $\Mcal_\gamma$ denote the space, modulo ambient isometries, of minimal embeddings to $\Sph^3$ of the closed surface of genus $\gamma$.  
For $\gamma =0$ or $\gamma = 1$, this space is completely understood: Almgren showed \cite{Almgren} the only minimal $2$-spheres in $\Sph^3$ are round, 
and Brendle \cite{BrendleLawson} proved Lawson's conjecture that the Clifford Torus is the only minimal surface with genus $1$ embedded in $\Sph^3$.  
In general, Lawson showed \cite{Lawson} that $\Mcal_\gamma$ is nonempty, Choi-Schoen proved $\Mcal_\gamma$ is compact \cite{ChoiSchoen} in the $C^k$ topology, 
and Bryant showed that every closed Riemann surface admits a conformal minimal immersion to $\Sph^4$.  
For large $\gamma$, constructions of minimal surfaces by doubling the equatorial $2$-sphere $\Spheq$ \cite{KPS, KapSph, KapMcG, LDG, KKMS} 
and by doubling the Clifford Torus \cite{PittsR, KY, Wiygul, LDG, Ketover} have each led to a better understanding of the moduli spaces $\Mcal_\gamma$.  

In particular, by constructing doublings of the Clifford torus, Ketover \cite[Theorem 1.15]{Ketover} showed the \emph{almost linear} bound 
$|\Mcal_\gamma| \geq C\gamma / \log \log \gamma$ for some $C> 0$ and all large enough $\gamma$, thus showing that $|\Mcal_\gamma|$ tends to infinity as $\gamma$ does.  
Very recently, Karpukhin-Kusner-McGrath-Stern proved \cite[Theorem 1.1]{KKMS} the \emph{linear} bound $|\Mcal_\gamma| \geq \lfloor \frac{\gamma-1}{4}\rfloor+1$, 
valid for all $\gamma \geq 0$, by constructing $\Spheq$-doublings by maximizing the normalized first eigenvalue of the Laplacian in an equivariant setting.  
In this article we prove a quadratic bound $|\Mcal_\gamma| \geq C \gamma^2$ by roughly counting the new Clifford torus doublings we produce.

The second problem we are interested in is the famous 
Yau conjecture for the first eigenvalue of minimal surfaces in $\mathbb{S}^3$ which was stated as follows. 

\begin{question}[Yau (1982), {\cite[Problem 100]{Yau:problems}}]
Is it true that the first eigenvalue for the Laplace-Beltrami operator on an embedded minimal hypersurface of $\Sph^{n+1}$ is $n$?
\end{question}
The affirmative answer to this question, which we refer to in this article as \emph{Yau's conjecture}, 
is known \cite{YangYau} to imply genus-dependent upper area bounds for embedded minimal surfaces in $\Sph^3$, 
and also to imply \cite{MontielRos} Lawson's conjecture that the Clifford torus is the unique genus-$1$ minimal surface embedded in $\Sph^3$.  
Independently of Yau's conjecture, Brendle proved  \cite{BrendleLawson} Lawson's conjecture in 2012.  
Furthermore, minimal embeddings in $\Sph^3$ by first eigenfunctions, which conjecturally comprise all minimal embeddings to $\Sph^3$ by Yau's conjecture, 
have been extensively studied \cite{LiYau, MontielRos, KusnerMcGrath, KKMS}; in particular, many new such examples have been recently constructed in \cite{KKMS}.

There are two main types of partial results towards Yau's conjecture.  
The first of these is to establish a suboptimal lower bound for the first eigenvalue $\lambda_1(M)$ on each minimal hypersurface $M$ in $\Sph^n$, 
and the second is to prove the conjectured equality $\lambda_1(M) = n$  for some subclass of minimal hypersurfaces in $\Sph^n$.  
With regards to the first type, in 1983, Choi-Wang \cite{ChoiWang} proved the universal lower bound $\lambda_1(M) \geq n/2$.  
Recently refinements of their method have led to improved bounds \cite{SpruckSire, Sire2} of the form $\lambda_1(M) \geq n/2 +C(M)$, 
where $C(M)$ is an explicit constant depending on $\max_M |A|$ tending to zero as $\max_M |A|$ tends to infinity; here $|A|$ is the length of the second fundamental form on $M$. 

With regards to results of the second type, Yau's conjecture is known to hold  \cite{TangYan} on all isoparametric minimal surfaces in $\Sph^{n+1}$, 
and by work of Choe-Soret \cite{ChoeSoret}, on the Lawson and Karcher-Pinkall-Sterling surfaces in $\Sph^3$.  
The Choe-Soret method is based on using Courant's nodal domain theorem in conjunction with the large symmetry groups generated by hyperplane reflections on the minimal surfaces under consideration.  

Kusner-McGrath extended this method and verified in \cite{KusnerMcGrath} Yau's conjecture for families of minimal $\Spheq$-doublings constructed in \cite{KapMcG}.  
Actually, the results in \cite{KusnerMcGrath}, which apply also to the Lawson and Karcher-Pinkall-Sterling surfaces, 
prove the stronger statement that the first eigenspace $\Ecal_{\lambda_1}(M)$ coincides with the span $\Coord(M)$ 
of the coordinate functions when $M$ is one of preceding minimal surfaces.  
Nonetheless, there are minimal surfaces with large symmetry groups, even generated by hyperplane reflections, 
for which these methods notably fail.  
This is the case for other $\Spheq$-doublings constructed in \cite{LDG}, and for Clifford Torus doublings \cite{KY, Wiygul, Ketover, LDG, KetoverCat}.  
In particular until now the Yau conjecture had not been verified for any Clifford Torus doublings. 
In this article we verify the Yau conjecture for our Clifford torus doublings and doublings of $\Spheq$ with small enough catenoidal bridges. 

\subsection*{Brief discussion of the results}
\phantom{ab}

This article involves two applications of our general methodology \cite{LDG} for constructions of minimal surface doublings in $3$-manifolds.  First, we use our general existence theorem \cite[Theorem A]{LDG} to construct many new minimal doublings of the Clifford Torus in $\Sph^3$; and second, we show that Yau's conjecture holds for many of the doublings in $\Sph^3$ constructible via this theorem, including for all such doublings of the equatorial two-sphere $\Spheq$, and for all Clifford Torus doublings constructed in the first part of the article.  These doublings encompass the prior doublings constructed via gluing methods.  We now discuss these results in more detail. 

In Theorem \ref{Tcliff}, we construct a family of doublings which is expected to include the examples constructed by Ketover in \cite{Ketover}.  The family is parametrized by a relatively prime pair $(a,b)$ of integers and positive integers $k$ and $m$; on each of $k$ maximally separated parallel copies of an $(a,b)$-torus knot on $\T$, the set $L$ of centers of the catenoidal bridges contains $m$ equally spaced points, arranged so that the symmetry group $\group^L_{\sym} \leq O(4)$ of $L$ as a subset of $\Sph^3 \subset \R^4$ acts transitively on $L$.  A fundamental domain for the action of $\group^L_{\sym}$ on $\T$ can be taken to be a parallelogram containing one point of $L$.   The construction can be carried out as long as $m$ is chosen large enough in terms of $k \sqrt{a^2+b^2}$ and $(a, b, k) \notin \{(1,1,1), (1,2,1), (2, 1, 1)\}$.  Given $(a, b, k)$, for many sufficiently large $m$, the corresponding doublings do not pass to a genus two surface in any lens space quotient.

For fixed but large $m$ any such doubling with parameters $(a, b, k)$ and $k=1$ has genus $m+1$.  Estimating the number of such doublings leads to (Proposition \ref{Pnum})  a bound of the form
\begin{align*}
|\Mcal_\gamma| \geq C \gamma^2
\end{align*}
on the number $|\Mcal_\gamma|$ of pairwise non-isometric minimal surfaces embedded in $\Sph^3$ with genus $\gamma$.  This improves Ketover's bound $|\Mcal_\gamma| \geq C \gamma/ \log \log \gamma$ from \cite{Ketover}, and the linear bound $|\Mcal_{\gamma}| \geq \lfloor \frac{\gamma-1}{4} \rfloor + 1$ from \cite{KKMS}.

While Ketover's doublings from \cite{Ketover} are expected to coincide with a subfamily of the surfaces from \ref{Tcliff}, we lack a proof that this is the case.  We remark, however, that characterizations of other families of minimal surfaces in $\Sph^3$ by their symmetry have recently been established \cite{KapWiygulsym, KWcliff}.

In Theorem \ref{Tcliff2}, we generalize the preceding to situations in which a fundamental parallelogram for the group $\group^L_{\sym}$ contains three points of the configuration $L$ which are inequivalent modulo $\group^L_{\sym}$.  As discussed in \cite[p. 305]{LDG}, these constructions can be further generalized to situations with more than three non-equivalent bridges per fundamental domain. 

We now discuss the results related to Yau's conjecture.  In Theorem \ref{Tsph} we prove that the first Laplace eigenspace $\Ecal_{\lambda_1}(M)$ on any side-symmetric minimal $\Spheq$-doubling $M$ in $\Sph^3$ constructible by our general existence result \cite[Theorem A]{LDG} is spanned by the coordinate functions, proving the Yau conjecture for such surfaces.  By side-symmetric we mean the doubling is invariant under the reflection of $\Sph^3$ fixing pointwise the base $\Spheq$; this condition is satisfied by all known $\Spheq$-doublings and is enforceable in \cite[Theorem A]{LDG}, and we expect all minimal $\Spheq$-doublings to be side-symmetric.  Importantly, and in contrast to the main results in \cite{ChoeSoret, KusnerMcGrath}, no other symmetry assumptions are required in Theorem \ref{Tsph}. 

In Theorem \ref{Tcliff} we prove the first Laplace eigenspace $\Ecal_{\lambda_1}(M)$ on any minimal $\T$-doubling $M$ in $\Sph^3$ constructed earlier in the article is spanned by the coordinate functions, proving Yau's conjecture for these doublings.  As mentioned before, the $\T$-doublings constructed in this article are expected to encompass all known such examples, including the ones constructed in the articles \cite{KY, Wiygul, LDG, Ketover}. 

\subsection*{Outline of strategy and main ideas}
\phantom{ab}

The minimal surface constructions follow the Linearized Doubling (LD) methodology introduced by NK in \cite{KapSph}, which we developed further in \cite{KapMcG, LDG}.  Given a minimal surface $\Sigma$ embedded in a Riemannian $3$-manifold $N$, the LD approach produces an embedded minimal surface $\Mbreve \subset N$ doubling $\Sigma$ in three steps.

The first step is to construct on $\Sigma$ a suitable family of \emph{Linearized Doubling (LD) solutions}: an LD solution $\varphi$ is a singular solution of $\Lcal_\Sigma \varphi = 0$ with logarithmic singularities, where $\Lcal_\Sigma$ is the Jacobi operator on $\Sigma$.  In the second step the LD solutions are converted to approximately minimal ``initial surfaces" with the aid of chosen finite-dimensional obstruction spaces $\skernelv[L]$.  The initial surface $M$ corresponding to an LD solution $\varphi$ consists of catenoidal bridges joined smoothly to the graphs of $\varphi + \vunder$ and $- \varphi - \vunder$ for some $\vunder \in \skernelv[L]$, and each bridge is located in the vicinity of a singular point of $\varphi$ and its size is given by the strength of the logarithmic singularity of $\varphi$ at the point.  In the final step, one of the initial surfaces is perturbed to exact minimality, providing the new minimal surface.

For the new constructions in \ref{Tcliff} and \ref{Tcliff2}, the symmetry group $\group^L_{\sym}$ is generated by reflections through three points $p_0, p_1, p_2$ on $\T$; with respect to a particular fundamental parallelogram $P$ for $\group^L_{\sym}$ (see \eqref{dL} and Figure \ref{Fim0}), these points correspond to a corner vertex of $P$, the midpoint of one side of $P$, and the midpoint of another side of $P$.  More specifically given $(a,b)$ as before, the geodesic segment joining $p_0$ and $p_1$ lies on an $(a,b)$ torus knot on $\T$.  In the first construction \ref{Tcliff}, there is one singularity modulo the symmetries, with the singular set consisting of the $\group^L_{\sym}$ orbit of one of the three points.  In the second construction \ref{Tcliff2}, there are three singularities modulo the symmetries, with the singular set the union of the $\group^L_{\sym}$-orbits of each of the three points.  See Figures \ref{Fim0} and \ref{Fim}.

In particular, the symmetries completely determine the singular sets $L$, and there are no horizontal forces.  In the first construction, the family of LD solutions has one continuous parameter, and in the second the family has three such parameters; in each case, the parameters govern the vertical matching equations corresponding to the points of $L$. 

We now turn to Yau's conjecture, and first describe the strategy for proving the first eigenspace $\Ecal_{\lambda_1}$ on a side-symmetric $\Spheq$-doubling  $\Mbreve$ embedded in $\Sph^3$ is spanned by the coordinate functions.  Because of the symmetry, we have an orthogonal decomposition $\Ecal_{\lambda_1} = \Ecal^+_{\lambda_1} \oplus \Ecal^-_{\lambda_1}$ into even and odd parts with respect to the side-exchanging reflection.  

Since a straightforward nodal-domain argument (Proposition \ref{Leven}) shows that $\Ecal^-$ is either trivial or spanned by the coordinate function vanishing on $\Spheq$, the main task is to establish (Proposition \ref{Pee}) the even part $\Ecal^+$ is either trivial or spanned by the remaining three coordinate functions. This is done by first decomposing the doubling $\Mbreve$ into \emph{catenoidal} and \emph{graphical} regions, and studying the contributions on each of these regions to the Dirichlet energy of an eigenfunction $u \in \Ecal^+_{\lambda_1}$.  Because of the even symmetry and the smallness of the bridges, it is shown in Lemma \ref{Lw0uest2} that the contribution from the catenoidal regions is small in comparison to the contribution from the graphical part.  Because the graphical part is quantitatively close to its projection $U\subset \Spheq$, the eigenfunction $u$ then gives rise to a function $w$ on $\Sigma$ with small Rayleigh quotient.  Finally, since the first eigenspace on $\Spheq$ is spanned by three coordinate functions, it follows that if $u$ is nonzero, it cannot be orthogonal to the span of these coordinate functions.

In the general case where the base minimal surface $\Sigma \subset \Sph^3$ is not totally-geodesic, a doubling $\Mbreve$ constructed from \cite[Theorem 5.7]{LDG} is not side-symmetric, so $\Ecal_{\lambda_1}$ cannot be decomposed into even and odd parts as before.  Nonetheless, $\Mbreve$ is approximately side-symmetric because its catenoidal bridges are small, and the two sides of $\Mbreve$ are exactly isometric with respect to a small perturbation $g'$ of the induced metric $g$.  Since the eigenspaces of $(\Mbreve, g')$ with low eigenvalues can be characterized similarly to the $\Spheq$-doubling case, the main task is to show that the smallest eigenvalue $\lambda_i$ for which $\Ecal^+_{\lambda_i} (\Mbreve, g')$ is nontrivial is strictly larger than $2$.  When the base $\Sigma$ is the Clifford Torus, and the doubling is constructed as in \ref{Tcliff} or in \ref{Tcliff2}, the additional symmetries enable us to estimate in Proposition \ref{Peigen} that $\lambda_i$ is close to $4$.

\subsection*{Organization of the presentation}
\phantom{ab}

Besides the introduction, this article has three more sections and one appendix.  In Section \ref{Scliff} we construct the new Clifford Torus doublings.  After defining the singular sets $L$ and symmetry groups in \ref{dL}, in \ref{Lphiave} and \ref{LPhipest} we estimate the corresponding LD solutions.  Theorems \ref{Tcliff} and \ref{Tcliff2} construct the families of doublings described earlier, and Proposition \ref{Pnum} proves the quadratic lower bound on the number of Clifford Torus doublings with prescribed genus. 

Section \ref{Ssph} is concerned with the first eigenspace $\Ecal_{\lambda_1} = \Ecal_{\lambda_1}(\Mbreve)$ on minimal doublings $\Mbreve$ of the equatorial $\Spheq$ in $\Sph^3$. After introducing a decomposition $\Ecal_{\lambda_1} = \Ecal^+_{\lambda_1} \oplus \Ecal^-_{\lambda_1}$ into even and odd parts with respect to the reflection fixing $\Spheq$, we characterize the odd part $\Ecal^-_{\lambda_1}$ in \ref{Cem} and the even part in \ref{Pee}.  These characterizations lead to the proof of the main result, Theorem \ref{Tsph}. 

In Section \ref{SYauCliff}, we study the first eigenspace $\Ecal_{\lambda_1}$ on doublings of the Clifford Torus doublings.  The main technical result is Proposition \ref{Peigen}, which provides a lower bound for an eigenvalue corresponding to an approximately odd eigenfunction on a such a doubling.  This leads to the main result, Theorem \ref{TcliffE}.  Finally, Appendix \ref{Stori} catalogues some well-known results \cite{Berger, MontielRos} concerning the spectrum of the Laplacian on two-dimensional tori. 

\subsection*{General notation and conventions}
\phantom{ab}

In comparing equivalent norms or other quantities we will find the following notation useful. 
\begin{definition}
\label{Dsimc}
We write $a\Sim_c b$ to mean that 
$a,b\in\R$ are nonzero of the same sign, 
$c\in(1,\infty)$, 
and $\frac1c\le \frac ab \le c$.  We also write $a \lesssim b$ to mean that $a \leq C b$ for some constant $C> 0$ depending only on the background $(\Sigma, N, g)$. 
\end{definition}

\subsection*{Acknowledgments}
\phantom{ab}

This article is partially based upon work supported by the National Science Foundation under Grant No. DMS-1928930, 
while NK was in residence at the Simons Laufer Mathematical Sciences Institute (formerly MSRI) in Berkeley, California, 
during the Fall 2024 semester.

\section{Notation and Background from \cite{LDG}}
\label{SLDG}

\subsection*{Elementary geometry and notation} 
\phantom{ab}

We introduce some notation from \cite{KapWiygulsym}.  

\begin{notation}
For any $A \subset \Sph^3 \subset \R^4$ we denote by $\Span(A)$ the span of $A$ as a subspace of $\R^4$ and we set $\Sph(A) : = \Span(A) \cap \Sph^3$. 
\end{notation}

Given a vector subspace $V$ of the Euclidean space $\R^4$, we denote by $V^\perp$ its orthogonal complement in $\R^4$, and we define the reflection in $\R^4$ with respect to $V$, $\Rcapunder_V : \R^4 \rightarrow \R^4$, by
\begin{align}
\Rcapunder_V : = \Pi_V - \Pi_{V^\perp},
\end{align}
where $\Pi_V$ and $\Pi_{V^\perp}$ are the orthogonal projections of $\R^4$ onto $V$ and $V^\perp$ respectively.

\begin{definition}[$A^\perp$ and reflections $\Rcapunder_A$]
\label{dref}
Given and $A \subset \Sph^3 \subset \R^4$, we define $A^\perp : = (\Span(A))^\perp \cap \Sph^3$ and $\Rcapunder_A : \Sph^3 \rightarrow \Sph^3$ to be the restriction to $\Sph^3$ of $\Rcapunder_{\mathrm{Span}(A)}$.  Occasionally we will use simplified notation: for example for $p, q \in \Sph^3$ we may write $\Rcapunder_{p, q}$ instead of $\Rcapunder_{\{p, q\}}$.  
\end{definition}

If $\group$ is a group acting on a set $B$ and if $A$ is a subset of $B$, then we refer to the subgroup
\begin{align}
\Stab_{\group}(A) : = \{ \gbold \in \group : \gbold A = A \}
\end{align}
as the \emph{stabilizer} of $A$ in $\group$.  For $A$ a subset of $\Sph^3$ we set
\begin{align}
\group^A_{\sym} : = \Stab_{O(4)} A = \{ \gbold \in O(4) : \gbold A = A\}. 
\end{align}

\begin{definition}
\label{dT}
Given $v = (a,b) \in \R$, we define $\Tcap_v \in SO(4)$ by 
\begin{align*}
\Tcap_v(z, w) = (e^{i a} z, e^{ib} w),
\end{align*}
where we have identified $\R^4$ with $\C^2$. 
\end{definition}

\subsection*{Definitions and Background from \cite{LDG}}
\phantom{ab}

In order to keep this article self-contained, we must introduce some notation and definitions from our article \cite{LDG}.  Our main goal is to state, in Theorem \ref{Ttheory}, the main theorem from \cite{LDG}, along with its assumptions, which we collect in \ref{ALDfam}.  To state these assumptions, which will be verified in Section \ref{Scliff} once we have described the specific setting in which we will work, we now introduce several definitions from sections 3 through 5 of \cite{LDG}.

In what follows, we let $\Sigma$ be the Clifford torus and $(N, g)$ to be the three-sphere $\Sph^3$ with its round metric, and denote by $\Lcal_\Sigma : = \Delta + |A|^2 = \Delta + 4$ the Jacobi operator on $\Sigma$.  

For $A \subset \Sigma$, we write $\dbold^{\Sigma, g}_A$ for the distance function from $A$ with respect to $g$ and we define the \emph{tubular neighborhood of $A$ of radius} $\delta > 0$ by $D_A(\delta) = \{ p \in \Sigma : \dbold^{\Sigma, g}_A(p) < \delta\}$.  We may omit either of $\Sigma$ or $g$ when clear from context. If $A$ is finite we may just enumerate its points in both cases, for example if $A = \{ q\}$ we write $\dbold_q(p)$.  

\begin{convention}
\label{cLker}
We assume we are given a finite set $L\subset \Sigma$ such that the operator $\Lcal_\Sigma$ has trivial kernel, when restricted to $\group^L_\sym$-symmetric functions. 
\end{convention}

\begin{definition}[Spaces of affine functions] 
\label{DVcal}
Given $p\in\Sigma$ let $\val[p]\subset C^\infty(T_p\Sigma) $ be the space of \emph{affine functions on $T_p\Sigma$}.  
Given a function $v$ which is defined on a neighborhood of $p$ in $\Sigma$ and is differentiable at $p$ we define $\Ecalunder_p v:= v(p)+ d_pv\in\val[p]$. 
$\forall\kappaunder\in\val[p]$ let $\kappaunder=\kappaperp+\kappa$ be the unique decomposition with $\kappaperp\in\R$ and $\kappa\in T^*_p\Sigma$  
and let $|\kappaunder| := |\kappaperp| + |\kappa|$.  
We define for later use $\val[L] := \bigoplus_{p\in L} \val[p]$ for any finite $L\subset\Sigma$.  
\end{definition} 

\begin{convention}
\label{con:alpha}
\label{Akappa} 
We fix now some $\alpha >0$ which we will assume as small in absolute terms as needed.  
\end{convention}

\begin{definition}[LD solutions]
\label{dLD}
We call $\varphi$ a \emph{linearized doubling (LD) solution on $\Sigma$}     
when there exists 
a finite set $L\subset \Sigma$, 
called the \emph{singular set of $\varphi$}, 
and a function 
$\taubold: L \rightarrow \R\setminus\{0\}$,  
called the \emph{configuration of $\varphi$}, 
satisfying the following,   
where $\tau_p$ denotes the value of $\taubold$ at $p\in L$.  
\begin{enumerate}[label=\emph{(\roman*)}]
\item 
$
\varphi \in C^\infty (\, \Sigma \setminus L \, ) 
$ 
and       
$\Lcal_\Sigma\varphi=0$ on $\Sigma\setminus L$.
\item 
$\forall p \in L$ 
the function   
$\varphi - \tau_p \log \dbold_p$ is bounded on some deleted neighborhood of $p$ in $\Sigma$. 
\end{enumerate}
\end{definition}

\begin{convention}[The constants $\delta_p$]
\label{con:L}
Given $L$ as in \ref{dLD} 
we assume that for each  $p\in L$ a constant $\delta_p>0$ has been chosen so that the following are satisfied. 
\begin{enumerate}[label={(\roman*)}]
\item 
$\forall p,p'\in L$ with $p\ne p'$ we have 
$D^\Sigma_p(9\delta_p) \cap D^\Sigma_{p'}(9\delta_{p'}) = \emptyset $. 
\item $\min_{p \in L} \delta_p$ is small enough in absolute terms. 
\end{enumerate} 
\end{convention}
Combined with the homogeneity of the Clifford Torus $\Sigma$ \ref{con:L}(ii) ensures that conditions (ii) and (iii) of \cite[Convention 3.8]{LDG} hold.  These conditions assert that $\delta_p$ is small enough in terms of the injectivity radius of the normal exponential map for $\Sigma$ in $N$ at $p$, and that the Dirichlet problem for $\Lcal_\Sigma$ is uniquely solvable on each $D_p(\delta_p)$, and are trivially verified in the present case where $\Sigma$ is the Clifford torus. 

\begin{assumption}[Obstruction spaces] 
\label{aK}
Given $L$ as in \ref{dLD} we assume we have chosen a subspace $\skernelv[L] =  \bigoplus_{p\in L}\skernelv[p] \subset C^\infty(\Sigma)$ satisfying the following, 
where the map 
$\Ecal_L : \skernelv[L] \rightarrow \val[L]$ (recall \ref{DVcal}) is defined by  
$\Ecal_L(v) := \bigoplus_{p\in L} \Ecalunder_p v$. 
\begin{enumerate}[label = {(\roman*)}]
\item The functions in $\skernelv[p]$ are supported on $D^\Sigma_p(4 \delta_p)$.
\item The functions in $\skernel[p]$, where $\skernel[p]: = \Lcal_\Sigma \skernelv[p]$, are supported on $D^\Sigma_p(4 \delta_p) \setminus D^{\Sigma}_p(\delta_p/4)$. 
\item 
$\Ecal_L : \skernelv[L] \rightarrow \val[L]$ is a linear isomorphism.  
\item  $\left\| \Ecal^{-1}_L \right\| \le C \delta_{\min}^{-2-\beta}$, 
where 
$\delta_{\min}:=\min_{p\in L}\delta_p$ 
and 
$\left\| \Ecal^{-1}_L \right\|$ is the operator norm of $\Ecal^{-1}_L  : \val[L] \rightarrow  \skernelv[L]$ 
with respect to the $C^{2, \beta}\left( \Sigma, g\right)$ norm  on the target and the maximum  norm on the domain subject to the metric $g$ on $\Sigma$. 
\item 
$\forall \kappaunderbold = (\kappaunder_p)_{p\in L} \in \val[L]$ we have for each $p\in L$
\[ 
\| \kappaunder_p \circ (\exp^\Sigma_p)^{-1} - \Ecal^{-1}_L\kappaunder_p : C^k( D^\Sigma_p(\delta_p), \dbold^\Sigma_p, g, (\dbold^\Sigma_p)^2)\| \le 
C(k) \, |\kappaunder_p| . 
\] 
\end{enumerate}
\end{assumption}

\begin{convention}[Uniformity of LD solutions] 
\label{con:one}
We assume given $\varphi$, $L$, and $\taubold$ as in \ref{dLD} with $\tau_p>0$ $\forall p\in L$, 
and $\delta_p$'s as in \ref{con:L}, 
satisfying the following with $\alpha$ as in \ref{con:alpha} and 
\begin{equation} 
\label{Ddeltaprime}
\begin{gathered} 
\tau_{\min}:=\min_{p\in L}\tau_p, 
\\ 
\delta_p':=\tau_p^\gammagl \quad (\forall p\in L),  
\end{gathered} 
\qquad  \qquad  
\begin{gathered} 
\tau_{\max}:=\max_{p\in L}\tau_p, 
\\ 
\delta_{\min}':=\min_{p\in L}\delta_p'=\tau_{\min}^\gammagl. 
\end{gathered} 
\end{equation} 
\begin{enumerate}[label=(\roman*)]
\item 
\label{con:one:i} 
\ref{con:L} holds and---in accordance with \ref{con:alpha}---$\tau_{\max}$ is as small as needed in terms of $\alpha$ only.  
\item
$\forall p\in L$ we have $9 \delta_p' = 9 \tau_p^\gammagl < \tau_p^{\alpha/100} < \delta_p \, $. 
\item $\tau_{\max}\le \tau_{\min}^{1-\gammagl/100}$.  
\item $\forall p\in L$ we have $(\delta_p)^{-2} \| \, \varphi : C^{2,\beta}(\, \partial D^\Sigma_p(\delta_p)    ,\, g\,)\,\| \le\tau_p^{1-\gammagl/9}$. 
\item $\| \varphi:C^{3,\beta} ( \, \Sigma  \setminus\disjun_{q\in L}D^\Sigma_q(\delta_q')    \, , \, g \, ) \, \|
\le
\tau_{\min}^{8/9} \, $.
\item
On $\Sigma\setminus\disjun_{q\in L}D^\Sigma_q(\delta_q') $ we have $\tau_{\max}^{1+\alpha/5} \le \varphi$.  
\end{enumerate}
\end{convention}

\begin{assumption}[Families of LD solutions] 
\label{ALDfam}
We assume \ref{cLker} holds and that we are given continuous families of the following parametrized by $\zetabold\in \domzb \subset \Pcal$, 
where $\Pcal$ is a finite dimensional vector space and $\domzb \subset \Pcal$ a convex compact subset containing the origin $\zerobold$. 
\begin{enumerate}[label=(\roman*)]
\item 
\label{Idiffeo} 
Diffeomorphisms $\Fcal^\Sigma_\zetabold : \Sigma \rightarrow \Sigma$ with $\Fcal^\Sigma_\zerobold$ the identity on $\Sigma$.  
\item 
Finite sets $L=L\llbracket \zetabold\rrbracket=\Fcal^\Sigma_\zetabold L\llbracket\zerobold\rrbracket \subset\Sigma$ of cardinality    
$|L|=|L\llbracket \zetabold\rrbracket|= |L\llbracket\zerobold\rrbracket|$.     
\item 
Configurations $\taubold=\taubold\llbracket\zetabold\rrbracket: L\llbracket\zetabold\rrbracket\to \R_+$.  
\item 
LD solutions $\varphi =\varphi \llbracket \zetabold \rrbracket$ as in \ref{dLD}  
of singular set 
$L=L\llbracket \zetabold\rrbracket$     
and configuration 
$\taubold=\taubold\llbracket\zetabold\rrbracket$.   
\item 
For each $L=L\llbracket \zetabold\rrbracket$  
constants $\delta_p=\delta_p\llbracket\zetabold\rrbracket$ as in \ref{con:L}.  
\item 
For each $L=L\llbracket \zetabold\rrbracket$  
a space $\skernelv \llbracket \zetabold \rrbracket = \skernelv[ \, L\llbracket \zetabold \rrbracket \, ]$ as in \ref{aK}. 
\item 
Linear isomorphisms $Z_{\zetabold}: \val\llbracket \zetabold\rrbracket \rightarrow \Pcal$ where 
$\val\llbracket \zetabold\rrbracket := \val [\, L \llbracket \zetabold \rrbracket \, ]$ (recall \ref{DVcal}).  
\end{enumerate} 
Moreover we assume the following are satisfied   
$\forall \zetabold\in \domzb$.      
\begin{enumerate}[label=\emph{(\alph*)}]
\item 
\label{Aa} 
$\| \Fcal^\Sigma_\zetabold : C^4 \| \le C$ where the norm is defined with respect to some atlas of $\Sigma$ and the constant $C$ depends only on the background $(\Sigma,N,g)$. 
\item 
\label{Ab} 
$\forall p\in L\llbracket\zerobold\rrbracket$ we have $\Fcal^\Sigma_\zetabold( D^\Sigma_p(3\delta_p)) = D^\Sigma_{q}(3\delta_q)$ with $q:=\Fcal^\Sigma_\zetabold(p)$. 
\item 
\label{Ac} 
$\varphi =\varphi \llbracket \zetabold \rrbracket$,  
$L=L\llbracket \zetabold\rrbracket$, and 
$\taubold=\taubold\llbracket\zetabold\rrbracket$ 
satisfy \ref{con:one}, including the smallness of $\tau_{\max}$ in \ref{con:one}\ref{con:one:i}
which is now in terms of the constant $C$ in \ref{Aa} as well.   
\item 
\label{Atau}
$\forall p\in L\llbracket\zerobold\rrbracket$ we have the uniformity condition $\tau^2_q \leq \tau_p \leq \tau_q^{1/2}$, 
where here $\tau_p$ denotes the value of $\taubold\llbracket\zerobold\rrbracket$ at $p$ 
and $\tau_q$ the value of $\taubold\llbracket \zetabold\rrbracket$ at $q = \Fcal^\Sigma_\zetabold(p)\in L\llbracket \zetabold \rrbracket$. 
\item 
\label{AZ}
$\zetabold - Z_{\zetabold} ( \Mcal_{L\llbracket\zetabold\rrbracket} \varphi\llbracket \zetabold\rrbracket) \in \frac{1}{2} \domzb$ (prescribed unbalancing). 
\end{enumerate} 
\end{assumption}

We are now ready to state the main existence result from \cite{LDG}. 

\begin{theorem}[Theorem A, \cite{LDG}] 
\label{Ttheory}
Assuming that 
\ref{ALDfam} holds,  
there is a smooth closed embedded minimal surface $\Mhat$ doubling $\Sigma$ in $N$ satisfying   
\begin{equation} 
\begin{gathered} 
\label{EMguO}
\Mhat = \graph^{N, g}_{\Omegahat}( \ubreve^+) \cup \graph^{N, g}_{\Omegahat} (-\ubreve^-),  
\\ 
\quad \text{where} \quad 
\Omegahat = \Pi_\Sigma(\Mhat) = \Sigma \setminus \textstyle{ \bigsqcup_{p \in L} \Dbreve_{p} }, 
\quad 
L = L\llbracket \zetaboldhat\rrbracket, 
\\
\text{ and } \quad 
D^\Sigma_p(\tau_p(1-\tau_p^{8/9})) \subset \Dbreve_p \subset D^\Sigma_p(\tau_p(1+\tau^{8/9}_p)) \quad \forall p\in L. 
\end{gathered} 
\end{equation} 
Moreover 
$\Mhat $ has genus $2g_\Sigma-1+|L|$ (where $g_\Sigma$ is the genus of $\Sigma$) and its area $|\Mhat|$ satisfies 
\begin{align}
\label{E:Marea} 
| \Mhat | = 2 |\Sigma| - \pi \sum_{p \in L } \tau^2_p\left(1+ O(\, \tau_p^{1/2}|\log \tau_p|\, )\right). 
\end{align}
\end{theorem}

\section{New Doublings of the Clifford Torus}
\label{Scliff}

\subsection*{Torus knots and symmetries}
\phantom{ab}

Let $\T: = \{(z_1, z_2) \in \C^2 : |z_1|=|z_2| = 1/\sqrt{2}\} \subset \Sph^3 \subset \C$ be the Clifford torus in the unit three-sphere $(\Sph^3, g)$.

\begin{definition}[Torus knots]
\label{dknot} 
Given $v=(a,b) \in \N \times \N$, define $\gamma = \gamma[v] \subset \T$ by
\begin{align}
\label{Eknot}
\gamma = \{ \frac{1}{\sqrt{2}} (e^{ia t}, e^{ibt}) : t \in \R\} 
= \{ \Tcap_{tv} p_0 : t \in \R\}, 
\quad
\text{where} 
\quad
p_0 = \frac{1}{\sqrt{2}} (1,1).
\end{align}
\end{definition}

\begin{assumption}
\label{Aab}
We will assume that the greatest common divisor of $a$ and $b$ is $1$. 
\end{assumption}

\begin{lemma}[Properties of $\gamma$]
\label{Lgammad}
The following hold.
\begin{enumerate}[label = \emph{(\roman*)} ]
\item The maximum value attained by $2\dbold^{\T}_\gamma$ is $\pi\sqrt{2}/|v|$; equivalently, the injectivity radius of the normal exponential map of $\gamma$ is $r_v:= \pi /( \sqrt{2} |v|)$.
\item The length of $\gamma$ is $|\gamma| = \pi \sqrt{2} |v|$.
\item $\group^\gamma_{\sym} = \langle  \Rcapunder_{p, \nu(p)}, \Rcapunder_{\gamma} | \, p \in \gamma \rangle$; moreover $\Rcapunder_{\gamma} = \id_{\Sph^3}$ if and only if $|v|>\sqrt{2}$. 
\end{enumerate}
\end{lemma}
\begin{proof}
Items (i) and (ii) follow from \eqref{Eknot}, \ref{Aab}, and elementary trigonometry, so we omit the details. 

 For (iii), note first that $\gamma$ lies fully in a $2$-dimensional subspace of $\R^4$ if $|v| = \sqrt{2}$, and lies fully in $\R^4$ if $|v|> \sqrt{2}$, which implies the second statement in (iii).  Now let $p \in \gamma$, and recall from \ref{dref} that $\Rcapunder_{p, \nu(p)}$ is the element of $O(4)$ which restricts to the identity on $\mathrm{Span}(p, \nu(p))$ and minus the identity on the orthogonal complement, which may be identified with $T_p \T$.  From this, it is easy to see that $\Rcapunder_{p, \nu(p)}(\T) = \T$ and that $\Rcapunder_{p, \nu(p)}$ restricts to the reflectional isometry of the square torus $\T$ through the point $p$.  In particular, $\Rcapunder_{p, \nu(p)}(\gamma) = \gamma$.
 
Denoting by $G$ the group $\langle  \Rcapunder_{p, \nu(p)}, \Rcapunder_{\gamma} | \, p \in \gamma \rangle$, the preceding shows that $G \leq \group^\gamma_{\sym}$, and it is easy to see that $G$ acts transitively on $\gamma$.

Now let $\gbold \in \group^\gamma_\sym$ be arbitrary.  By the preceding, there is an element $\gbold' \in G$ such that $\gbold \gbold'$ fixes $p$, preserves the orientation of $\gamma$, and preserves the directions normal to $\gamma$.  It follows that $\gbold \gbold'$ is the identity on $\T$.  Finally, since any isometry of $\Sph^3$ which fixes $\T$ pointwise must be the identity, it follows that $\gbold \gbold'$ is the identity, hence $\group^\gamma_{\sym} \leq G$.
\end{proof}

\begin{assumption}
\label{Am}
We fix $k, m \in \N$ and assume $m$ is as large as needed in terms of $k$ and $v$. 
\end{assumption}
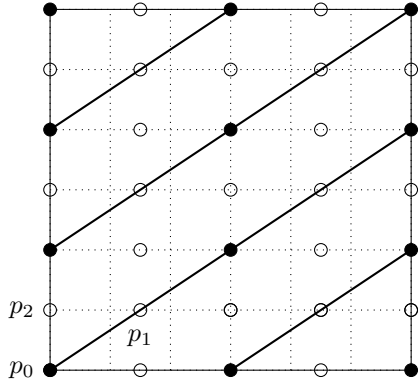
\begin{figure}[h]
	\centering
	\begin{tikzpicture}[scale=.8]
	 \draw[thin,dotted] (0,0) grid (6,6);
	\draw[thin] (0, 0) rectangle (6,6); 
	\draw[thick] (0, 0) to (6, 4); 
	\draw[thick] (3, 0) to (6, 2); 
	\draw[thick] (0, 2) to (6,6);
	\draw[thick] (0, 4) to (3, 6);
	\draw (0, 0) node[circle,  left]{$p_0$}; 
	\draw (1.5, 1) node[circle, below ]{$p_1$}; 
	\draw (0, 1) node[circle,  left]{$p_2$};
	\filldraw[color=black](0, 0) circle (3pt);
	\draw[](1.5, 1) circle (3pt);
	\draw[](0, 1) circle (3pt);
	\filldraw[](3, 2) circle (3pt);
	\filldraw[](6, 4) circle (3pt);
	\filldraw[](0, 4) circle (3pt);
	\filldraw[](3, 6) circle (3pt);
	\filldraw[](3, 0) circle (3pt);
	\filldraw[](6, 2) circle (3pt);
	\filldraw[](0, 2) circle (3pt);
	\filldraw[](3, 4) circle (3pt);
	\filldraw[](6, 6) circle (3pt);
	\filldraw[](0, 0) circle (3pt);
	\filldraw[](0, 6) circle (3pt);
	\filldraw[](6, 0) circle (3pt);
	\draw[](3, 1) circle (3pt);
	\draw[](4.5, 1) circle (3pt);
	\draw[](6, 1) circle (3pt);
	\draw[](3, 1) circle (3pt);
	\draw[](4.5, 1) circle (3pt);
	\draw[](6, 1) circle (3pt);
	\draw[](0, 3) circle (3pt);
	\draw[](1.5, 3) circle (3pt);
	\draw[](3, 3) circle (3pt);
	\draw[](4.5, 3) circle (3pt);
	\draw[](6, 3) circle (3pt);
	\draw[](0, 5) circle (3pt);
	\draw[](1.5,5) circle (3pt);
	\draw[](3, 5) circle (3pt);
	\draw[](4.5, 5) circle (3pt);
	\draw[](6,5) circle (3pt);
	\draw[](1.5, 0) circle (3pt);
	\draw[](4.5,0) circle (3pt);
	\draw[](1.5, 2) circle (3pt);
	\draw[](4.5,2) circle (3pt);
	\draw[](1.5, 4) circle (3pt);
	\draw[](4.5,4) circle (3pt);
	\draw[](1.5, 6) circle (3pt);
	\draw[](4.5,6) circle (3pt);
	\end{tikzpicture}
	\caption{A depiction of a torus knot $\gamma$ and singular set $L$ as in \ref{dL} with $v = (2, 3)$, $k=1$ and $m = 6$.  The points of $L$ are denoted by black circles, and points $p$ of $\T \setminus L$ for which the reflection $\Rcapunder_{p, \nu(p)}$ about $p$ preserves $L$ are denoted by white circles. }
		\label{Fim0}
\end{figure}

With $p_0 \in \T$ as in \eqref{Eknot}, define a collection $L_0 \subset \gamma$ of $m$ equally spaced points on $\gamma$ by
\begin{align*}
L^0 = \{ \Tcap_{\frac{2\pi jv}{m}} p_0 : j=0, \dots, m-1 \}.
\end{align*}
Next, the set $\{ w' \in \R^2 \setminus \mathrm{Span}(v) : \Tcap_{w'} p_0 \in L^0\}$ is clearly nonempty and discrete, and hence contains an element $w$ with minimal length.  Now define a subset $L \subset \T$ of $km$ points, distributed on $k$ parallel copies of $L^0$,  points $p_1, p_2 \in \T$, and a symmetry group $\group$ by 
\begin{equation}
\label{dL}
\begin{gathered}
L = L[k,m,v] := \bigcup_{j =0}^{k-1} L^0, 
\quad
\text{where for }
j\geq 1,
\quad
L^j := \Tcap_{\frac{jw}{k}} L^0, 
\\
p_1 = \Tcap_{\frac{\pi v}{m}} p_0, 
\quad
p_2 = \Tcap_{\frac{w}{2k}} p_0, 
\quad
\group = \group[k, m, v]: = \langle \Rcapunder_{p_i, \nu(p_i)} |\,  i =1,2, 3\rangle.
\end{gathered}
\end{equation}
Clearly $L \subset \Lpar$, where $\Lpar \subset \T$ is a union of $k$ copies of $\gamma$ defined by
\begin{align}
\label{ELpar}
\Lpar = \bigcup_{j=0}^{k-1} \Tcap_{\frac{jw}{k}} \gamma.
\end{align}
Note by \ref{dT} and \ref{Lgammad}(ii) that the distance between adjacent components of $\Lpar$ is
 $2 r_v/k = \pi \sqrt{2}/ k|v|$.

\begin{lemma}[Properties of $\group^L_{\sym}$]
\label{Lgroup}
The following hold. 
\begin{enumerate}[label = \emph{(\roman*)} ]
\item For each $p\in L$, $\group^{L}_{\sym}$ contains the reflection $\Rcapunder_{p, \nu(p)}$ about the point $p$.
\item $\group^L_{\sym}$ acts transitively on $L$. 
\item $\group^L_{\sym} = \group$ if $|v| > \sqrt{2}$. 
\end{enumerate}
\end{lemma}
\begin{proof}
By \ref{dL}, it is easy to see that the reflection $\Rcapunder_{p, \nu(p)}$ of $\T$ about a point $p\in L$ preserves $L$, so (i) follows.  Similarly, it is easy to see $\Rcapunder_{p_i, \nu(p_i)}$ preserves $L$ as a set, which implies $\group \leq \group^L_{\sym}$, and implies (ii).

For (iii), let $\gbold \in \group^{L}_{\sym}$ be arbitrary.  Since $\group$ acts transitively on $L$, it follows that there is an element $\gbold' \in \group$ such that $\gbold \gbold'$ fixes $p$; moreover, since $|v|> \sqrt{2}$, we may further assume $\gbold \gbold'$ induces the identity map on $T_p \T$.  It follows that $\gbold \gbold'$ is the identity on $\T$.  Finally, since any isometry of $\Sph^3$ which fixes $\T$ pointwise must be the identity, it follows that $\gbold \gbold'$ is the identity, hence $\group^L_{\sym} \leq \group$. 
\end{proof}

\begin{remark}
When $|v|\leq \sqrt{2}$, that is, when one component of $v$ is zero or when $|v| = \sqrt{2}$, $\group^L_{\sym}$ clearly contains $\group$ as a proper subgroup and in particular contains the reflection fixing the geodesic joining $p_0$ to $p_i$ for $i=1, 2$.  The constructions from \cite[Theorem 6.17]{LDG} have $|v| = 1$.

Furthermore, when $|v| =  \sqrt{2}$, $\group^L_{\sym}$ may contain the binary dihedral group as a subgroup; Pitts-Rubinstein proposed examples of this form in \cite[Example 11]{PittsR}.
\end{remark}

\begin{notation}
If $X$ is a function space consisting of functions defined on a domain $\Omega \subset \T$ and $\Omega$ is invariant under the action of $\groupL$, we use a subscript ``sym" to denote the subspace $X_{\sym} \subset X$ consisting of those functions $f \in X$ which are invariant under the action of $\groupL$.
\end{notation}

\begin{lemma}
\label{Lker}
The following hold.
\begin{enumerate}[label = \emph{(\roman*)} ]
\item $(\ker \Lcal_\T)_{\sym}$ is trivial, where $\Lcal_\T= \Delta_\T + 4$ is the linearized operator on $\T$. 
\item There is a unique $\group^L_{\sym}$-symmetric LD solution (in the sense of \cite[Definition 3.4]{LDG})  $\Phi = \Phi[L]$ with singular set $L$ and satisfying $\tau_p = 1$ $\forall p \in L$. 
\end{enumerate}
\end{lemma}
\begin{proof}
Take geodesic normal coordinates $(x,y)$ for $\T$ centered at $p_0 \in L$; in these coordinates the space $\Lcal_\T$ is spanned by the functions $\sin 2 x, \sin 2y, \cos 2x$, and $\cos 2y$.

Denote by $f$ the restriction of $\Rcapunder_{p_0, \nu(p_0)}$ to $\T$ and note that $f(x,y) = (-x,-y)$ in the local coordinate system above.  If $u$ is in the span of $\sin 2x$ and $\sin 2y$, it is not difficult to see that $u \circ f = - u$, hence
\begin{align*}
\int_\T uv = \int_\T (u \circ f) (v \circ f) = - \int_\T u v
\end{align*}
for any $v \in L^2_{\sym}(\T)$. 

Next, repeating the same argument at $p_1 \in L$ shows that the analogous functions in geodesic normal coordinates centered at $p_1$ are orthogonal to $L^2_{\sym}(\T)$.  After using the trigonometric addition formula $\sin (a+b) = \sin a \cos b + \sin b \cos a$ to write these in terms of $\sin 2 x, \sin 2y, \cos 2x$, and $\cos 2y$ and using the largeness of $m$, we conclude that $\cos 2x$ and $\cos 2y$ are orthogonal to $L^2_{\sym}(\T)$.  This completes the proof of (i). 

Item (ii) follows immediately from (i) and \cite[Lemma 5.22]{LDG}.
\end{proof}

\subsection*{Analysis of $\group^L_{\sym}$-symmetric LD Solutions}
\phantom{ab}

By \ref{Lker}, there is a unique $\group^L_{\sym}$-symmetric LD solution $\Phi = \Phi[k, m, v]$ with singular set $L$ and satisfying $\tau_p = 1$ $\forall p \in L$.  We now estimate $\Phi$.

\begin{definition}
\label{dave}
Given a function $\varphi$ defined on some domain $\Omega \subset \T$, we define a function $\varphi_{\ave}$ on the union $\Omega'$ of the parallel curves of $\gamma$ on which $\varphi$ is integrable (whether contained in $\Omega$ or not), by requesting that on each such curve $C'$, 
\begin{align*}
\varphi_{\ave}|_{C'} : = \avg_{C'}\varphi.
\end{align*}
we also define $\varphi_{\osc}$ on $\Omega \cap \Omega'$ by $\varphi_{\osc}: = \varphi - \varphi_{\ave}$.  Finally, we call a function $\varphi$ defined on a domain $\Omega \subset \T$ which is a union of parallel circles of $\gamma$ and depends only on $\dbold_{\Lpar}$ \emph{rotationally invariant}. 
\end{definition}

We next characterize the rotationally invariant part $\Phi_\ave$ of $\Phi$.  The following lemma and its proof generalize \cite[Lemma 6.8]{LDG}, which covered the case where $|v| = 1$. 

\begin{lemma}[The rotationally invariant part $\Phi_\ave$]
\label{Lphiave}
The following hold. 
\begin{enumerate}[label = \emph{(\roman*)} ]
\item $\Phi_\ave = \frac{m}{\sqrt{8} | v| \sin \frac{\sqrt{2}\pi}{k|v|}} \cos ( \sqrt{2} \frac{\pi}{k|v|}- 2 \dbold^\T_{\Lpar})
$.  
\item $\Phi_\ave |_{\Lpar} = \frac{m}{2F}$, where $F: = \sqrt{2} |v|  \tan\frac{\sqrt{2} \pi}{k|v|}$.

\end{enumerate}
\end{lemma}
\begin{proof}
Since $\Lcal_\T \Phi_{\ave} = 0$ on $ \T \setminus \gamma$ and the maximum value of $2\dbold^{\T}_{\Lpar}$ is $\sqrt{2}\pi/k|v|$ by \ref{Lgammad}, the symmetries imply that 
$\Phi_\ave = C \cos ( \sqrt{2}\pi/k|v|- 2 \dbold_{\Lpar})$
 for some $C \neq 0$. By integrating $\Lcal_\T \Phi = 0$ on $\Omega_{\epsilon_1, \epsilon_2} : = D^{\T}_{\Lpar}(\epsilon_2)\setminus D^{\T}_{L}(\epsilon_1),$ where $0<\epsilon_1 << \epsilon_2$ and integrating by parts, we obtain
\begin{align*}
\int_{\partial\Omega_{\epsilon_1, \epsilon_2}} \frac{\partial}{\partial \eta} \Phi + \int_{\Omega_{\epsilon_1, \epsilon_2}} 4 \Phi = 0,
\end{align*}
where $\eta$ is the unit outward conormal field along $\partial \Omega_{\epsilon_1, \epsilon_2}$.  By taking the limit as $\epsilon_1 \searrow 0$ and then as $\epsilon_2 \searrow 0$, we obtain by using the logarithmic behavior near $L$, the preceding, and \ref{Lgammad}(ii) that
\begin{align*}
2\pi k m = 4 \pi \sqrt{2} k |v| C \sin \frac{\sqrt{2}\pi}{k|v|},
\end{align*}
which implies the conclusion. 
\end{proof}

For convenience, we define a scaled metric and scaled linear operator by
\begin{align}
\label{Egtilde}
\gtilde = m^2 g, 
\quad
\Lcaltilde_\T = \frac{1}{m^2} \Lcal_\T
= \Delta_{\gtilde} + \frac{4}{m^2}.
\end{align}
We also define $\delta = 1/(100m)$.

\begin{definition}
\label{Dphat}
Define $\Ghat \in C^\infty_{\sym}( \T \setminus L)$ and $\Phat, \Phip, E' \in C^\infty_{\sym}(\T)$ by requesting that
\begin{equation*}
\begin{gathered}
\Ghat = \Psibold[ 2 \delta, 3 \delta; \dbold^\T_p] ( G_p - \log \delta \cos (2 \dbold^\T_{\Lpar}), 0) 
\quad
\text{on} 
\quad
D^\T_L(3 \delta), 
\\
\Phat = \Phi_\ave - \Psibold \left[ \frac{2}{m}, \frac{3}{m}; \dbold^\T_{\Lpar}\right] \left( \frac{m}{\sqrt{8}|v|} \sin ( 2 \dbold^\T_{\Lpar}), 0\right), 
\end{gathered}
\end{equation*}
that $\Ghat = 0$ on $\T \setminus D^\T_L (3 \delta)$, $\Phat = \Phi_\ave$ on $\T \setminus D^\T_{\Lpar}(3/m)$, and
\begin{align}
\label{Ephat}
\Phi = \Ghat + \Phat + \Phip, 
\quad
E' = - \Lcaltilde_{\T}(\Ghat + \Phat). 
\end{align}
\end{definition}

\begin{remark}
Note that from Lemma \ref{Lphiave} and the fact that
\begin{align*}
\cos ( \sqrt{2} \frac{ \pi}{k|v|}- 2 \dbold^\T_\gamma) 
= 
\cos \frac{ \sqrt{2}\pi}{k|v|} \cos 2 \dbold^\T_\gamma
+
\sin \frac{\sqrt{2}\pi}{k|v|} \sin 2 \dbold^\T_\gamma
\end{align*}
that $\Phat$ as defined in \ref{Dphat} is indeed smooth across $\gamma$.
\end{remark}

\begin{lemma}
\label{LPhipest}
$E'$ vanishes on $D^\T_L (2\delta)$ and $E'_\osc$ is supported on $D^\T_{\Lpar}(3 \delta)$.  Moreover:
\begin{enumerate}[label = \emph{(\roman*)} ]
\item $\| \Ghat : C^j_{\sym}(\T \setminus D^\T_L(\delta), \gtilde) \| \leq C(j)$. 
\item $\| E' : C^j_{\sym}(\T, \gtilde)\| \leq C(j)$.
\item $\| \Phip : C^j_{\sym}(\T, \gtilde)\| \leq C(j)$. 
\end{enumerate}
In (ii), the same estimate holds if $E'$ is replaced with either $E'_\ave$ or $E'_\osc$. 
\end{lemma}
\begin{proof}
The proofs of items (i) and (ii) are almost identical to the proofs of the corresponding items in \cite[Lemma 6.12]{LDG}, so we omit the details.  To prove (iii), it suffices to prove that the estimate holds when $\Phip = \Phip_{\ave} + \Phip_{\osc}$ is replaced by either $\Phip_{\ave}$ or $\Phip_{\osc}$.  We first prove the estimate for $\Phip_{\ave}$.  Note by \ref{Dphat} that
\begin{align*}
\Phip_{\ave} = \frac{m}{\sqrt{8}|v|} \sin ( 2 \dbold^{\T}_{\Lpar}) - \Ghat_{\ave}.
\end{align*} 
Note that the left hand side is smooth and the discontinuities on the right hand side cancel.  Using that $\Lcaltilde_{\T} \Phip_{\ave} = E'_{\ave}$, on $D^{\T}_{\Lpar}(3/m)$ we have 
\begin{align}
\label{Ephipode}
\partial^2_{\, \ssstilde} \Phip_{\ave} + \frac{4}{m^2} \Phip_{\ave} = E'_{\ave},
\end{align}
where $\ssstilde : = m \dbold^{\T}_{\Lpar}$.  On a neighborhood of $\partial D^{\T}_{\Lpar}(2/m)$, we have that $\Ghat = 0$ from Definition \ref{Dphat}.  It follows that $| \Phip_{\ave}| < C$ and $| \partial_{\, \sss} \Phip_\ave| < C$ on $\partial D^{\T}_{\Lpar} (2/m)$. Using this as initial data for the ODE and bounds on the inhomogeneous term from (ii) yields the $C^2$ bounds on $\Phip_\ave$ in (iii).  Higher derivative estimates follow inductively from differentiating \eqref{Ephipode} and again using (ii).

We now estimate $\Phip_{\osc}$.  By taking oscillatory parts of the equation \eqref{Ephat} defining $\Phip$, we have
\begin{align*}
\Lcaltilde_{\T} \Phip_{\osc} = E'_{\osc},
\end{align*}
where we have used that $\Lcaltilde_{\T} \Phip_{\osc} = ( \Lcaltilde_{\T} \Phip)_{\osc}$.
Now consider the quotient surface $\T / \group^L_{\sym}$; it is a torus, and in the metric induced from $\gtilde$ has shorter side bounded from below by a positive constant.   From this and Appendix \ref{Stori}, it follows that the first eigenvalue of $\Lcaltilde_{\T}$ with respect to the $\gtilde$ metric, when restricted to functions with average zero, is bounded from below.  By standard theory, it then follows that
\begin{align*}
\| \Phip_{\osc} : C^{2, \beta}_{\sym}(\T, \gtilde)\| 
\leq 
C
\| E' : C^{0, \beta}_{\sym}( \T, \gtilde)\|. 
\end{align*}
Using this, standard regularity theory, and (ii), we conclude that
\begin{align*}
\| \Phip_{\osc} : C^j_{\sym}( \T, \gtilde)\| \leq C(j).
\end{align*}
Item (iii) then follows using that $\Phip = \Phip_{\ave} + \Phip_{\osc}$ and the preceding estimates. 
\end{proof}

\begin{theorem} 
\label{Tcliff}
Let $v = (a,b) \in \N\times \N$ be a relatively prime pair of integers, $\gamma$ be an $(a,b)$-torus knot in the Clifford torus $\T \subset \Sph^3$, $\Lpar \subset \T$ be a maximally separated collection of $k \in \N$ parallel copies of $\gamma$ as in \ref{ELpar}, $L \subset \Lpar$ a collection of $km$ points as in \ref{dL}, and $\group^L_{\sym} \leq O(4)$ be the group of isometries preserving the set $L$.  There is a constant $C> 0$ such that if $m> C k |v|$ and $(a, b, k) \notin \{ (1, 1, 1), (1, 2, 1), (2, 1, 1)\}$,  there exists a smooth embedded $\group^L_{\sym}$-symmetric minimal doubling $\Mbreve$ of $\T$ in $\Sph^3$ with genus $km+1$ and $km$ doubling holes, each containing a point of $L$.

Moreover, there is a number $\cunder> 0$ independent of $v, k, m$ such that the area $|\Mbreve|$ of $\Mbreve$ satisfies
\begin{align*}
|\Mbreve| = 2 |\T| - \frac{1}{m} e^{2\zeta'} \exp\bigg(- \frac{m}{\sqrt{2}|v| \tan \frac{\pi \sqrt{2}}{k|v|}}\bigg),
\end{align*}
for some constant $\zeta' = \zeta'(v, k, m) \in [-\cunder, \cunder]$.
\end{theorem}
\begin{proof}
We will apply the general existence result \cite[Theorem A]{LDG} for minimal doublings, with the obvious simplifications and adaptations necessary in order to accommodate the symmetries.  In order to apply this result, we must check that Assumption 5.2 of \cite{LDG} holds.

To this end, let $v = (a, b), \gamma, m$, and $L$ be as in the statement of the theorem.  Recall by \ref{Lker} that the kernel $(\mathrm{ker} \Lcal_{\T})_{\sym}$ of the linearized operator when restricted to $\group^L_{\sym}$-symmetric functions is trivial, so Assumption 4.1 of \cite{LDG} holds.  

We now define continuous families of objects as in \cite[Assumption 5.2]{LDG} parametrized by $\zeta \in [-\cunder, \cunder] $, where $\cunder> 0$ is a constant independent of $m$ which will be chosen later, and check that that items (i)-(vii) of Assumption 5.2 hold.  Specifically, we define diffeomorphisms $\Fcal^\T_{\zeta} : \T \rightarrow \T$, finite sets $L[ \zeta] \subset \T$, configurations $\taubold[\zeta] : L[\zeta]\rightarrow \R_+$, LD solutions $\varphi[ \zeta ]$ with singular sets $L[\zeta]$ and configurations $\taubold [\zeta] : L[\zeta] \rightarrow \R_+$, and for each $L[\zeta]$, constants $\delta_p[\zeta]$ by
\begin{equation}
\begin{gathered}
\label{Edefs}
\Fcal^\T_\zeta = \id_\T, 
\quad
L[\zeta] = L \quad \text{(recall \ref{dL})},
\quad
\tau_p[\zeta] = \tau: = \frac{1}{m}e^{\zeta} e^{-\frac{m}{2F}} \quad \text{for } p \in L, 
\\
\varphi [\zeta ] = \tau \Phi,
\quad 
\text{and}
\quad
\delta_p[\zeta] : = \delta = 1/(100m),
\end{gathered}
\end{equation}
where $F=\sqrt{2}|v| \tan \frac{\sqrt{2} \pi}{k |v|}$ was defined in \ref{Lphiave}.
It follows from these definitions that items (i)-(v) of Assumption 5.2 hold. 

 Next, define obstruction spaces $\skernelv \llbracket \zeta \rrbracket$ as in \cite[Remark 3.12]{LDG}; from the definition and the preceding, \cite[Assumption 3.11]{LDG} holds, hence item (vi) holds.  Now define symmetric versions of the obstruction spaces by letting $\skernelv_{\sym}[L], \val_{\sym}[L]$ be the subspaces of $\skernelv[L], \val[L]$ consisting of the $\group^L_{\sym}$-invariant elements, where $\skernelv[L] = \bigoplus_{p \in L} \skernelv[p]$.  Since $\group^L_{\sym}$ acts transitively on $L$, contains $\Rcap_{p, \nu(p)}$ for each $p \in L$ (recall \ref{Lgroup}), and the differential of any $\Rcapunder_{p, \nu(p)}$-invariant function vanishes at $p$, it is easy to see that $\val_{\sym}[L]$ is $1$-dimensional and may be identified with $\R$.

For (vii), define a family of linear isomorphisms $Z_\zeta  : \val_{\sym} \llbracket \zeta \rrbracket \rightarrow \Pcal$ by $Z_\zeta(\mu ) = \frac{1}{\tau} \mu$. 

Next we check that items (a)-(e) of Assumption 5.2 hold: (a), (b), and (d) are trivial from the prior definitions and the symmetry.  For (c) we must check the uniformity conditions \cite[Convention 3.15]{LDG} on the LD solutions $\varphi[\zeta]$.  The assumption on $(a, b, k)$ implies that $F = \sqrt{2}|v| \tan \frac{\sqrt{2}\pi}{k |v|}$ defined in \ref{Lphiave} is positive, hence
 $\tau$ in \eqref{Edefs} can be made as small as needed by taking $m$ large.  Then 3.15(i)-(iii) follow by taking $m$ large enough.  Again using that $F>0$, it follows from \ref{Lphiave} that $\Phi_{\ave} > ckm$ for some $c>0$.  The estimates in 3.15(iv)-(vi) now follow easily using that $\varphi = \tau \Phi$, the decomposition of $\Phi$ in \ref{Dphat}, and the estimates on $\Ghat$ and $\Phi'$ in \ref{LPhipest}.  This completes the verification of Assumption 5.2(c).

Finally, we check that item (e) of Assumption 5.2 holds.  By the symmetry, it suffices to show that
\begin{align}
\label{EZbd}
\zeta - Z_\zeta (\Mcal_p  \varphi) \in \frac{1}{2}[-\cunder, \cunder]
\end{align}
for any point $p \in L$, where we recall that $Z_\zeta \mu = \mu / \tau$. 

To prove \eqref{EZbd} holds, recalling the definition of the mismatch $\Mcal_p \varphi \in \val[p]$ from \cite[Definition 3.10]{LDG}, in light of the $\group^L_{\sym}$-symmetry, we see that $\Mcal_p \varphi$ can be identified with the real number such that
\begin{align*}
\varphi \circ \exp_p^{\T}(w) = \tau \log (2 |w|/\tau) + (\Mcal_p\varphi)(w) + O(|w|^2\log |w|)
\end{align*}
for small $w \in T_p \T$.  In particular using this and \eqref{Ephat} to expand $\Mcal_p \varphi$, we find 
\begin{align*}
\frac{1}{\tau} \Mcal_p \varphi 
= \frac{m}{2F} + \log  \frac{\tau}{2\delta} + \Phip(p)
= \zeta + \Phip(p) + \log (50), 
\end{align*}
where the second equality uses \eqref{Edefs}.  Equation \eqref{EZbd} now follows using \ref{LPhipest}(iii) to estimate $\Phip(p)$ and taking $\cunder$ large enough in terms of $\log 50$ and the bound on $\Phip$ from \ref{LPhipest}.  This completes the verification of \cite[Assumption 5.2]{LDG}.

We may therefore apply \cite[Theorem A]{LDG}, which implies the conclusions of the Theorem. 
\end{proof}

\begin{prop}
\label{Pnum}
There is a constant $c>0$ such that for all $m \in \N$ sufficiently large, the number of pairwise nonisometric minimal $\T$-doublings from Theorem \ref{Tcliff} with genus $m+1$ is at least $cm^2$. 
\end{prop}
\begin{proof}
Let $C> 0, v = (a,b), k$, and $m$ be as in the statement of Theorem \ref{Tcliff}, so that as long as $m\geq Ck|v|$, the theorem ensures the existence of an embedded $\T$-doubling $\Mbreve = \Mbreve[v, k, m]$ with $m$ bridges on each of $k$ parallel $(a, b)$ torus knots.  Fixing $m$ large and denoting by $\Mbreve_{v, m} = \Mbreve[v, 1, m]$, we are interested in the number of pairwise nonisometric doublings $\Mbreve_{v,m}$.  

Given $R>0$, let $S_R = \{ v = (a, b) \in \N\times \N : a< b, \mathrm{gcd}(a,b) = 1,  |v|\leq R\}$.  By classical number-theoretic arguments (see Theorems 331 and 332 in \cite{HardyWright}), there is a constant $C_1> 0$ such that for all large enough $R$, the cardinality of $S_R$ is at least $C_1 R^2$.  

By combining the preceding, in order to prove the theorem, it suffices to prove that if $v, v' \in S_{m/C}$, then $\Mbreve_{v, m}$ and $\Mbreve_{v', m}$ are not isometric.  By the definition of $S_R$, each $v \in S_R$ can be identified with a torus knot $\gamma[v]$ on $\T$ with slope $b/a> 1$.  Since $v$ and $v'$ correspond to different slopes, the corresponding knots $\gamma[v]$ and $\gamma[v']$ in $\T$ are not related by an isometry of $\T$.  Since the catenoidal bridges of $\Mbreve_{v, m}$ and $\Mbreve_{v', m}$ are positioned with centers equally spaced on $\gamma[v]$ and $\gamma[v']$, it follows from the largeness of $m$ that $\Mbreve_{v,m}$ and $\Mbreve_{v', m}$ are not related by an isometry of $\Mbreve$. 
\end{proof}

\subsection*{Configurations with three singularities modulo symmetries}
\phantom{ab}

In this subsection we construct and estimate $\group[k, m, v]$-symmetric LD solutions on $\T$ which have three singularities on each fundamental domain, and apply \cite[Theorem 5.7]{LDG} to construct corresponding minimal surfaces.  This generalizes the construction from \cite[Theorem 6.25]{LDG}, where we considered the case where $|v|=1$.

To simplify the estimates, we assume in this subsection that $m/k < C_1$ for a fixed constant $C_1>0$.  To begin, for $p_0, p_1, p_2, \group = \group[k,m,v]$ as defined in \eqref{dL}, define
\begin{align}
\label{ELup}
L = L[k, m , v] := \bigcup_{i=0}^2 L_i := \bigcup_{i=0}^2 \group p_i
\end{align}
and define for $i=0, 1, 2$ the $\group$-invariant LD solution $\Phi_i = \Phi_i[k,m, v]$ satisfying $\tau_p = 1$  $\forall p \in L_i$. 

\begin{figure}[h]
	\centering
	\begin{tikzpicture}[scale=.8]
	 \draw[thin,dotted] (0,0) grid (6,6);
	\draw[thin] (0, 0) rectangle (6,6); 
	\draw[thick] (0, 0) to (6, 4); 
	\draw[thick] (3, 0) to (6, 2); 
	\draw[thick] (0, 2) to (6,6);
	\draw[thick] (0, 4) to (3, 6);
	\draw (0, 0) node[circle,  left]{$p_0$}; 
	\draw (1.5, 1) node[circle, below ]{$p_1$}; 
	\draw (0, 1) node[circle,  left]{$p_2$};
	\filldraw[color=black](0, 0) circle (3pt);
	\draw[](1.5, 1) circle (3pt);
	\draw[](0, 1) circle (3pt);
	\filldraw[](0, 0) circle (3pt);
	\filldraw[](0, 1) circle (3pt);
	\filldraw[](0, 2) circle (3pt);
	\filldraw[](0, 3) circle (3pt);
	\filldraw[](0, 4) circle (3pt);
	\filldraw[](0, 5) circle (3pt);
	\filldraw[](0, 6) circle (3pt);
	\filldraw[](3, 0) circle (3pt);
	\filldraw[](3, 1) circle (3pt);
	\filldraw[](3, 2) circle (3pt);
	\filldraw[](3, 3) circle (3pt);
	\filldraw[](3, 4) circle (3pt);
	\filldraw[](3, 5) circle (3pt);
	\filldraw[](3, 6) circle (3pt);
	\filldraw[](6, 0) circle (3pt);
	\filldraw[](6, 1) circle (3pt);
	\filldraw[](6, 2) circle (3pt);
	\filldraw[](6, 3) circle (3pt);
	\filldraw[](6, 4) circle (3pt);
	\filldraw[](6, 5) circle (3pt);
	\filldraw[](6, 6) circle (3pt);
	\filldraw[](1.5, 1) circle (3pt);
	\filldraw[](1.5, 3) circle (3pt);
	\filldraw[](1.5, 5) circle (3pt);
	\filldraw[](4.5, 1) circle (3pt);
	\filldraw[](4.5, 3) circle (3pt);
	\filldraw[](4.5, 5) circle (3pt);
	\draw[](3, 1) circle (3pt);
	\draw[](4.5, 1) circle (3pt);
	\draw[](6, 1) circle (3pt);
	\draw[](3, 1) circle (3pt);
	\draw[](4.5, 1) circle (3pt);
	\draw[](6, 1) circle (3pt);
	\draw[](0, 3) circle (3pt);
	\draw[](1.5, 3) circle (3pt);
	\draw[](3, 3) circle (3pt);
	\draw[](4.5, 3) circle (3pt);
	\draw[](6, 3) circle (3pt);
	\draw[](0, 5) circle (3pt);
	\draw[](1.5,5) circle (3pt);
	\draw[](3, 5) circle (3pt);
	\draw[](4.5, 5) circle (3pt);
	\draw[](6,5) circle (3pt);
	\draw[](1.5, 0) circle (3pt);
	\draw[](4.5,0) circle (3pt);
	\draw[](1.5, 2) circle (3pt);
	\draw[](4.5,2) circle (3pt);
	\draw[](1.5, 4) circle (3pt);
	\draw[](4.5,4) circle (3pt);
	\draw[](1.5, 6) circle (3pt);
	\draw[](4.5,6) circle (3pt);
	\end{tikzpicture}
	\caption{A depiction of a torus knot $\gamma$ and singular set $L$ as in \ref{ELup} with $v = (2, 3)$, $k=1$ and $m = 6$.  The points of $L$ are denoted by black circles, and points $p$ of $\T \setminus L$ for which the reflection $\Rcapunder_{p, \nu(p)}$ about $p$ preserves $L$ are denoted by white circles. }
	\label{Fim}
\end{figure}
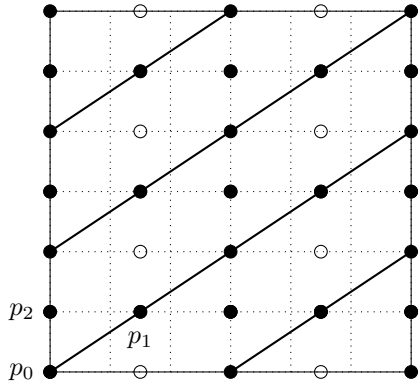

Define a space of parameters $\Pcal : = \R^2$, and for a constant $\cunder>0$ to be determined later, define $B_{\Pcal}: = [-\cunder, \cunder]\times [ - \frac{\cunder}{km}, \frac{\cunder}{km}]$, and for each $\zetabold = (\zeta, \sigma_1, \sigma_2) \in B_\Pcal$ and LD solution
\begin{align}
\label{Ephithree}
\varphi = \varphi \llbracket \zetabold \rrbracket : = \sum_{i=0}^2 e^{\sigma_i}  \tauunder \Phi_i,
\quad 
\text{where}
\quad
\tauunder = \tauunder \llbracket \zetabold \rrbracket : = \frac{1}{m} e^{\zeta} e^{-3 \frac{km}{4\pi}},
\end{align}
and by convention we define $\sigma_0: = - \sigma_1 - \sigma_2$. 

Since $\group^L_{\sym}$ acts transitively on each $L_i$, $i=1,2,3$ and $\forall p \in L$, the differential of any $\Rcapunder_{p, \nu(p)}$-invariant function vanishes at $p$, it follows that $\val_{\sym}\llbracket \zetabold \rrbracket := \val_{\sym} [L]$ is $3$-dimensional and may be identified with $\R^3$.

\begin{prop}
\label{Pzb}
There is an absolute constant $C$ (independent of $\cunder$) such that for $k, m$ as in \ref{Am}, $mk$ large enough (depending on $\cunder$), and $m/k< C_1$ the map $Z_\zetabold : \val_{\sym}\llbracket \zetabold \rrbracket \rightarrow \Pcal$ defined by
\begin{align}
\label{Zzeta}
Z_{\zetabold}(\muboldtilde) = \frac{1}{3}\left( \sum_{i=0}^2 \mutilde_i, \frac{4\pi}{3km} (\mutilde_0 + \mutilde_2 - 2 \mutilde_1), 
\frac{4\pi}{3km} (\mutilde_0 + \mutilde_1 - 2 \mutilde_2)
\right),
\end{align}
where here $\muboldtilde = \tauunder ( \mutilde_0, \mutilde_1, \mutilde_2)$ satisfies $\zetabold - Z_\zetabold( \Mcal_L \varphi) \in [- C, C] \times [ \frac{-C}{km} , \frac{C}{km} ]^2$.
\end{prop}
\begin{proof}
Using \ref{Lphiave}, \ref{Dphat},  \eqref{Ephithree}, and the largeness of $k$ in the assumption,  for each $i \in \{0,1,2\}$ it follows that $\mu_i : = \frac{1}{\tauunder} \Mcal_{p_i} \varphi $ satisfies
\begin{equation}
\label{Emu}
\begin{aligned}
\mu_i &= \frac{km}{4\pi} \sum_{j=0}^2 e^{\sigma_j} + \bigg( \sum_{j=0}^2 e^{\sigma_j} \Phi'_j\bigg) \left. \right|_{ p_i } + e^{ \sigma_i} \log \frac{e^{\sigma_i} \tauunder}{2\delta} + O(1/k)\\
&= 3 \frac{km}{4\pi} + O\left( \frac{\cunder^2}{km}\right) + O(C) + (1+ \sigma_i) \log \frac{\tauunder}{2\delta} +O(1/k),
\end{aligned}
\end{equation}
where we have expanded the exponentials and used that $\sum_{i=0}^2 \sigma_i = 0$.  Therefore,
\begin{align}
\label{Emusum}
\frac{1}{3}\sum_{i=0}^2 \mu_i = 3 \frac{km}{4\pi} + \log \frac{\tauunder}{2\delta} + O( C + \frac{\cunder^2}{km} ) = \zeta + O \left( C + \frac{\cunder^2}{km}\right).
\end{align}
Using \eqref{Emu}, that $\sum_{i=0}^2 \sigma_i = 0$, and \eqref{Ephithree}, we calculate
\begin{equation}
\label{Emufinal}
\begin{aligned}
\frac{1}{3}( 2 \mu_2 - \mu_1 - \mu_0) &= - 3 \frac{km}{4\pi} \sigma_2 + O \left( C+ \frac{\cunder^2}{km} \right),
\\
\frac{1}{3}( 2 \mu_1 - \mu_2 - \mu_0) &= - 3 \frac{km}{4\pi} \sigma_1 + O \left( C+ \frac{\cunder^2}{km} \right).
\end{aligned}
\end{equation}
The proof is concluded by combining \eqref{Emusum} and \eqref{Emufinal}.
\end{proof}

\begin{theorem}
\label{Tcliff2}
Given $C_1> 0$, there exists an absolute constant $\cunder>0$ such that for all $(k, m) \in \N^2$ , $m$ large enough in terms of $\cunder$, and $m/k < C_1$, there exists a genus $3mk +1$, $\group^L_{\sym}$-invariant doubling of $\T$ by applying \cite[Theorem 5.7]{LDG}. 
\end{theorem}
\begin{proof}
The proof is very similar to the proof of \ref{Tcliff} and consists of checking the hypotheses of \cite[Theorem 5.7]{LDG}, so we only give a sketch, pointing out some of the differences.  Although $\tau$ takes on three distinct values, Assumption 5.2(d) still holds because of \eqref{Ephithree}.  The map $Z_{\zetabold}$ defined in \eqref{Zzeta} is clearly a linear isomorphism for each $\zetabold \in B_{\Pcal}$, so by Proposition \ref{Pzb}, Assumption 5.2(e) holds. 
\end{proof}

\section{The first eigenspace for minimal doublings of $\Spheq$}
\label{Ssph}

\subsection*{General notation and conventions}
\phantom{ab}

Let $S$ be a closed surface embedded in the round three-sphere $(\Sph^3, g)$.  
Denote by $0= \lambda_0(S)< \lambda_1(S)< \cdots$ the eigenvalues of the Laplacian on $S$, by $\Ecal_{\lambda_i}(S)$ the eigenspace corresponding to $\lambda_i(S)$, and by $\Coord(S)$ the span of the coordinate functions on $S$.  If $S$ is minimal, the coordinate functions are well-known to be eigenfunctions with eigenvalue $2$, that is $\Coord(S) \subset \Ecal_{2}(S)$.  Throughout, we will tacitly use that the eigenspaces $\Ecal_{\lambda_i}(S)$ are mutually orthogonal with respect to the $L^2(S)$-inner product.

The goal of this section is to prove the following theorem, which proves the Yau conjecture for all  minimal $\Spheq$-doublings  constructible from our general existence result \cite[Theorem A]{LDG}.

\begin{theorem}
\label{Tsph}
Let $\Mbreve$ be a minimal doubling of an equatorial two-sphere $\Spheq$ in the round three-sphere $(\Sph^3, g)$ obtained from \cite[Theorem 5.7]{LDG}.  If $\Mbreve$ is side-symmetric, 
then the first nontrivial eigenspace of the Laplacian on $\Mbreve$ is spanned by the coordinate functions. 
\end{theorem}

\noindent By \emph{side-symmetric}, we mean that $\Mbreve$ is invariant under the reflection of $\Sph^3$ fixing $\Spheq$ pointwise. 

\begin{remark}
Note the following.
\begin{itemize}
\item The side-symmetry in Theorem \ref{Tsph} can be enforced in the application of \cite[Theorem 5.7]{LDG}; moreover, all known minimal $\Spheq$-doublings in $\Sph^3$ are side-symmetric.
\item Theorem \ref{Tsph} applies to the infinite families of $\Spheq$-doublings constructed in \cite{KapSph, KapMcG, LDG, KapZou}, but also more generally to any minimal doubling constructible using the general assumptions of \cite[Theorem 5.7]{LDG}; in particular, no symmetry beyond the side-symmetry already mentioned is required. 
\end{itemize}
\end{remark}

\begin{remark}
\label{AYau}
In the notation adopted above, Theorem \ref{Tsph} amounts to the assertion that $\Lapone(\Mbreve) = \Coord(\Mbreve)$.  Note also that $\Lapone(\Sph^2) = \Coord(\Sph^2)$ so in particular, $\lambda_1(\Sph^2) = 2$. 
\end{remark}

Before proving Theorem \ref{Tsph}, we fix some notation and conventions which will be used throughout. 

\begin{convention}
\label{con:eps}
Let $(\Sigma, N, g) = (\Spheq, \Sph^3, g)$ be the \emph{background} (recall the notation of \cite[Convention 2.1]{LDG}) consisting of an equatorial two-sphere $\Sigma = \Spheq$ embedded in the round three-sphere $(N, g) = (\Sph^3, g)$, $\Mbreve$ be a minimal $\Spheq$-doubling as in Theorem \ref{Tsph}, $\rho$ be the reflection of $\Mbreve$ fixing $\Mbreve \cap \Sph^2$, and $\epsilon >0$ be a number depending only on the background which may be taken as small as needed.

Finally, when the role of $\Mbreve$ is clear, we write $\Ecal_{\lambda_i}$ and $\Coord$ in place of $\Ecal_{\lambda_i}(\Mbreve)$ and $\Coord(\Mbreve)$. 
\end{convention}

In particular, the assumption that $\Mbreve$ is obtained using \cite[Theorem 5.7]{LDG} implies (recall  \cite[Convention 3.15(i)]{LDG}) that $\tau_{\max}$ may be taken as small as needed in terms of $\epsilon$.

\subsection*{Even-odd decompositions}
\phantom{ab}

 Because $\rho$ is an involutive isometry, it induces a linear involutive isometry $\rho^*$ of $L^2(\Mbreve)$ given by $\rho^*u =  u \circ \rho$.  Consequently, each invariant subspace $X \subset L^2(\Mbreve)$ for $\rho^*$ admits an $L^2(\Mbreve)$-orthogonal direct sum decomposition 
\begin{align}
\label{Epmdec}
X = X^+ \oplus X^-,
\quad
\text{where}
\quad
X^\pm  := \{ u \in X : \rho^* u  = \pm  u\}
\end{align}
 into subspaces of \emph{even} ($X^+$) and \emph{odd} ($X^-$) functions with respect to $\rho$. 
 
Using this notation, the conclusion of Theorem \ref{Tsph} is equivalent to the assertions that $\Ecal_{\lambda_1}^\pm = \Coord^\pm$.  The subspaces $\Ecal^+_{\lambda_1}$ and $\Ecal^-_{\lambda_1}$ will therefore be studied separately. 

\subsection*{The odd part $\Ecal^-_{\lambda_1}$}
\phantom{ab}

Because the two-sphere $\Sigma = \Sph^2 \subset \Sph^3$ lies in a $3$-dimensional subspace of $\R^4$, the subspace $\Coord^- \subset \Coord$ of odd coordinate functions is  $1$-dimensional.  We now show that $\Ccal^-$ contains $\Ecal^-_{\lambda_1}$. 
\begin{prop}
\label{Cem}
\label{Leven}
Either $\lambda_1 < 2$ and 
$\Lapone^- = \{ 0\}$, or
$\lambda_1 = 2$ and $\Ecal_{\lambda_1}^- = \Coord^-$.
\end{prop}
\begin{proof}
It clearly suffices to prove that $\Ecal^-_{\lambda_1} \subset \Coord^-$.  To this end, suppose $u \in \Lapone^-$ is nonzero.
  Since $u$ is odd, its nodal set contains $\Mbreve \cap \Sph^2$; furthermore, since $\Mbreve \cap \Sph^2$ separates $\Mbreve$ and $u$ is a first eigenfunction, the Courant nodal domain theorem \cite{Courant} implies its nodal set is equal to $\Mbreve \cap \Sph^2$.

On the other hand, the nodal set of any nonzero odd coordinate function $v \in \Coord^-$ is also equal to $\Mbreve \cap \Sph^2$.  Now let $M_+$ be one of the components of $\Mbreve \setminus \Mbreve \cap \Sph^2$.  Using the oddness of $u$ and $v$, we find
\begin{align*}
\int_{\Mbreve} u v d\sigma  = \int_{M_+} u v d\sigma + \int_{\rho(M_+)} u v d\sigma = 2 \int_{M_+} u v d\sigma \neq 0.
\end{align*}  
Therefore, $u$ and $v$ are not $L^2(\Mbreve)$-orthogonal, which implies that $\Ecal^-_{\lambda_1} \subset \Coord^-$. 
\end{proof}

\subsection*{The even part $\Ecal^+_{\lambda_1}$}
\phantom{ab}

The study of $\Ecal^+_{\lambda_1}$ uses a decomposition of $\Mbreve$ into \emph{catenoidal} and \emph{graphical} regions, whose projections to $\Sigma$ are defined below.  In the definition, $\alpha> 0$ denotes a small number (recall \cite[Convention 2.10]{LDG}) from the construction of $\Mbreve$ in \cite[Theorem 5.7]{LDG} which depends only on the background $(\Spheq, \Sph^3, g)$.

\begin{definition}
\label{dOmegap}
For each $p \in L$, define $\Omegap = D^\Sigma_p(\tau^{2\alpha}_p)$, $\Omega = \cup_{p \in L} \Omega_p$, and $U = \Sigma \setminus \Omega$. 
\end{definition}

Recall from \cite[Theorem 5.7]{LDG} that $\Mbreve$ is the union of the graphs of functions $\ubreve^+$ and $-\ubreve^-$, where the functions $\ubreve^\pm$ are defined on $\Sigmabreve : = \Sigma \setminus \sqcup_{p \in L} \Dbreve_p$ for $\Dbreve_p$ a small perturbation of the disk $D^\Sigma_p (\tau_p)$ and satisfy 
\begin{align}
\label{Eubreve}
\| \ubreve^\pm : C^{2, \beta}(U) \| \lesssim \tau^{8/9}_{\max}.
\end{align}
Note the symmetry assumption (recall \ref{con:eps}) implies $\ubreve^+  = \ubreve^-$. 
We denote by $X_{\pm}: \Sigmabreve \rightarrow \Mbreve$ the maps parametrizing the graphs of $\ubreve$ and $-\ubreve^-$.

\begin{definition}[The approximate eigenfunction $w$]
\label{dw0}
Given $u \in \Lapone^+$, 
define a piecewise smooth function $w$ on $\Sigma$ by requesting that
\begin{align}
\label{Ew0def}
w = \Xp^* u = \Xm^* u
\quad
\text{on} 
\quad
U
\quad
\text{and}
\quad
\Delta w = 0
\quad
\text{on} 
\quad
\Omega
\end{align}
and a decomposition $w = \wlow + \whigh$, where $\wlow$ is the projection of $w$ onto $\Ecal_{\lambda_0}(\Sigma)\oplus \Ecal_{\lambda_1}(\Sigma)$.
 \end{definition}
 
\noindent Note that $w$ is well-defined in \ref{dw0} because the evenness of $u$ and the symmetry imply $\Xp^* u = \Xm^* u$.

The next lemma shows that the $L^2$ norms of $w$ and $d w$ on the graphical region $U$ are well-approximated by the corresponding norms of $u$ and $d u$ on $X_{\pm}(U)$. 
\begin{lemma}
\label{Lw0uest1}
For $u$ and $w$ as in \ref{dw0}, the following hold. 
\begin{enumerate}[label=\emph{(\roman*)}]
	\item $\| w \|_{L^2( U)} \Sim_{1+C\epsilon} \| u\|_{L^2( \Xpm(U))}$. 
	\item $\| d w \|_{L^2( U)} \Sim_{1+C\epsilon} \| d u\|_{L^2(\Xpm(U))}$.
	\end{enumerate}
\end{lemma}
\begin{proof}
Recalling from \ref{dw0} that $\Xpm^* u = w$, we have
\begin{align*}
\| w \|^2_{L^2(U)} = \| u \|^2_{L^2 ( \Xpm(U))} + \int_{U} w^2 ( d\sigma - \Xpm^* d\sigma).
\end{align*}
Item (i) follows from this after using \eqref{Eubreve} to estimate the second integral.  The proof of (ii) is very similar and is omitted. 
\end{proof}

Next, we show that the $L^2$ norms of $u$ and $d u$ do not concentrate on the catenoidal regions $\Pi^{-1}_{\Sigma}(\Omega)$, and analogously that the $L^2$ norms of $w$ and $d w$ do not concentrate on $\Omega$.  
\begin{lemma}
\label{Lw0uest2}
For $u$ and $w$ as in \ref{dw0}, the following hold. 
\begin{enumerate}[label=\emph{(\roman*)}]
	\item $\| u \|_{L^2(\Pi^{-1}_\Sigma(\Omega))} \lesssim \epsilon \| u \|_{L^2(\Mbreve)}$. 
	\item $\| d u \|_{L^2(\Pi^{-1}_\Sigma(\Omega))} \lesssim \epsilon \| d u \|_{L^2(\Mbreve)}$. 
	\item $\| w \|_{L^2(\Omega)} \lesssim \epsilon \| w \|_{L^2(\Sigma)}$. 
	\item $\| d w  \|_{L^2(\Omega)} 
	 \lesssim \epsilon \| d w \|_{L^2(\Sigma)}$. 
	\end{enumerate}
\end{lemma}
\begin{proof}

Fix $p \in L$, define $\runder= \tau^{2 \alpha}_p$, $\rtop = \tau^{\alpha}_p$ and for any $r \in [\runder,  \rtop]$ domains  $ \cat_{r}$ and $ \cat(r)$ of $\cat=  \cat[p, \tau_p, \kappaunder_p] \subset M\llbracket \zetaboldhatunder\rrbracket$ by
\begin{align*}
\begin{gathered}
\cat_r =(X^{N, g}_{M, \upphihat})^{-1}(\Mhat \cap \Pi^{-1}_\Sigma ( D^\Sigma_p(r))),
\\
\cat(r) = \cat \cap \Pi_\Sigma^{-1} ( D^\Sigma_p(r)),
\end{gathered}
\end{align*}
where the catenoidal bridge $\cat$ was defined in \cite[Definition 2.14]{LDG}, the initial surface $M = M\llbracket \zetaboldhatunder\rrbracket$ and perturbation function $\upphihat$ are as in \cite[Theorem 5.7]{LDG}, and the parametrization $X^{N, g}_{M, \upphihat} : M \rightarrow \Mbreve$ is as in \cite[Notation 1.2(viii)]{LDG}.  Define also $\ucheck \in M$ by $\ucheck = (X^{N, g}_{M, \upphihat})^* u$, and let $\gpeuc: = (\exp^{\Sigma, N, g}_p)_* \, g|_p $ be the metric induced by the Euclidean metric on $T_p N$ through the Fermi exponential map $\exp^{\Sigma, N,g}_p$ (recall \cite[Definition 2.2]{LDG}).
Using the smallness of $\upphihat$ from \cite[Theorem 5.7]{LDG},  the corresponding closeness of $\cat(\runder)$ to $\cat_{\runder}$, and the smallness of $g - \gpeuc$ in \cite[Lemma 2.28]{LDG}, it follows \cite[Lemma C.12]{LDG} that
\begin{equation}
\label{Eucheck0}
\begin{gathered}
\| u \|_{L^2( \Pi^{-1}_\Sigma ( \Omegap ), g)}
 \Sim_{1+C\epsilon} 
 \| \ucheck \|_{L^2(\cat(\runder), \gpeuc)},
 \\ 
 \| \nabla u \|_{L^2( \Pi^{-1}_\Sigma ( \Omegap ), g)}
 \Sim_{1+C\epsilon} 
  \| \nabla \ucheck \|_{L^2(\cat(\runder), \gpeuc)}.
  \end{gathered}
\end{equation}

Now let $\ucir$ be the $\gcir$-harmonic function on $\cat(\rtop)$ with $\ucir = \ucheck$ on $\partial \cat(\rtop)$.  Then $\ucheck - \ucir$ satisfies
\begin{equation*}
\begin{cases}
\Delta_{\gpeuc} ( \ucheck - \ucir) 
=E
\hfill
&\text{on}
\quad \cat(\rtop),
\\
\ucheck - \ucir = 0
\hfill
&\text{on}
\quad \partial \cat(\rtop),
\end{cases}
\end{equation*}
where $E = ( \Delta_{\gpeuc} - \Delta_{\gcheck}) \ucheck- \lambda_1(\Mbreve)\ucheck $.
Then for any $r \in [\runder, \rtop]$, scaling the metric on $\cat(r)$ to be of unit size and estimating the smallness of the inhomogeneous term, we use standard elliptic theory to estimate $\ucheck - \ucir$ and conclude 
\begin{equation}
\label{Eucheck}
\begin{gathered}
\| \ucheck \|_{L^2(\cat(r), \gpeuc)}
\Sim_{1+C\epsilon}
\| \ucir \|_{L^2(\cat(r), \gcir)},
\\
\| d \ucheck \|_{L^2(\cat(r), \gpeuc)}
\Sim_{1+C\epsilon}
\| d \ucir \|_{L^2(\cat(r), \gcir)}.
\end{gathered}
\end{equation}

To estimate $\ucir$, consider the space $L^2_{\sym}( \cat, \gcir)$ and complete orthogonal set $\{v_k\}_{k \in \Z}$, where
\begin{align*}
L^2_{\sym}( \cat, \gcir) = \{v \in  L^2(\cat, \gcir) : \rho^*v = v\},
\quad
v_k(\sss, \theta) = \cosh(k \sss) e^{ik \theta},
\end{align*}
and $(\sss, \theta)$ are the standard coordinates on $\cat$ (recall \cite[Notation 2.3]{LDG}).  Now fix $k \in \Z$.  
A straightforward calculation using the largeness of $\rtop/ \runder = \tau^{-\alpha}_p$ shows
\begin{equation}
\label{Emodes}
\begin{gathered}
\| v_k \|_{L^2( \cat(\runder), \gcir )} \lesssim \epsilon \| v_k \|_{L^2( \cat(\rtop), \gcir)},
\\
\| d v_k \|_{L^2( \cat(\runder), \gcir )} \lesssim \epsilon \| d v_k \|_{L^2( \cat(\rtop), \gcir)}.
\end{gathered}
\end{equation}
Since $\rho^* \ucir = \ucir$, expanding $\ucir$ in terms of the $v_k$ and using the Pythagorean theorem with \eqref{Emodes} shows
\begin{equation}
\label{Eubr}
\begin{gathered}
\| \ucir \|_{L^2( \cat(\runder), \gcir )} \lesssim \epsilon \| \ucir \|_{L^2( \cat(\rtop), \gcir)},
\\
\| d \ucir \|_{L^2( \cat(\runder), \gcir )} \lesssim \epsilon \| d \ucir \|_{L^2( \cat(\rtop), \gcir)},
\end{gathered}
\end{equation}
and combining \eqref{Eucheck0}, \eqref{Eucheck}, and \eqref{Eubr} establishes
\begin{equation}
\begin{gathered}
\| u \|_{L^2(\Pi^{-1}_\Sigma(\Omegap), g)} 
\lesssim
\epsilon
\| u\|_{L^2(\Pi^{-1}_\Sigma(D^{\Sigma}_p(\tau_p^{\alpha})))},
\\
\| d u \|_{L^2(\Pi^{-1}_\Sigma(\Omegap), g)} 
\lesssim
\epsilon
\|d u\|_{L^2(\Pi^{-1}_\Sigma(D^{\Sigma}_p(\tau_p^{\alpha})))}.
\end{gathered}
\end{equation}
Since any two disks $D^\Sigma_p(\tau^\alpha_p), D^\Sigma_q(\tau^\alpha_q)$ are disjoint for distinct $p, q\in L$, items (i)-(ii) follow.

By very similar considerations, we have
\begin{align*}
\| w \|_{L^2(\Omegap, g)} 
\Sim_{1+C\epsilon} \| w \|_{L^2( \Omegap, \gpeuc)}
\Sim_{1+C\epsilon} \| \wcir \|_{L^2( \Omegap, \gpeuc)},
\end{align*}
where $\wcir$ is the $\gcir$-harmonic function on $\Omegap$ with $\wcir = w$ on $\partial \Omegap$,
and another straightforward estimate using separation of variables shows
\begin{align}
\| \wcir \|_{L^2( \Omegap)} 
\lesssim
\epsilon \| \wcir \|_{L^2( D^\Sigma_p( \tau_p^\alpha))}. 
\end{align}
It follows in the same way as in the proof of (i) that (iii) holds.  Finally, the proof of (iv) is extremely similar, so we omit the details. 
\end{proof}

Using the estimates established in \ref{Lw0uest1} and \ref{Lw0uest2}, we now establish an upper bound for the Rayleigh quotient of $w$. 

\begin{cor}
\label{Cray0}
For $u$ and $w$ as in \ref{dw0}, the following hold.
\begin{enumerate}[label=\emph{(\roman*)}]
\item $\| d w \|^2_{L^2(\Sigma)}
\leq 
(1+ C\epsilon) \lambda_1(\Sigma) \| w \|^2_{L^2(\Sigma)}.$
\item $|2 \langle w , v \rangle_{L^2(\Sigma)} 
	-  \langle u, v \rangle_{L^2(\Mbreve)}|
	 \lesssim \epsilon \| w \|_{L^2(\Sigma)}  \| v \|_{C^1(N)}$ for any $v \in C^\infty(N)$.
\item $\| \wlow\|_{L^2(\Sigma)}  \lesssim \epsilon\| w \|_{L^2(\Sigma)}$, if $u \in \Coord^\perp$. 
\end{enumerate}
\end{cor}
\begin{proof}
We first prove (i).  On one hand, we have  $\|\nabla u\|^2_{L^2(\Mbreve)} = \lambda_1(\Mbreve) \|u\|^2_{L^2(\Mbreve)}$ because  $u$ is a first eigenfunction.  On the other hand, combining  \ref{Lw0uest1} and \ref{Lw0uest2} implies
\begin{equation*}
\begin{gathered}
\| u \|^2_{L^2(\Mbreve)} \Sim_{1+C\epsilon} 2\| w \|^2_{L^2(\Sigma)},
\\ 
\| d u \|^2_{L^2(\Mbreve)} \Sim_{1+C\epsilon} 2\| d w \|^2_{L^2(\Sigma)}.
\end{gathered}
\end{equation*}
From these facts (i) follows using that $\lambda_1(\Mbreve) \leq  \lambda_1(\Sigma)=2$.

By the triangle inequality, to prove (ii), it suffices to prove
\begin{equation*}
\begin{gathered}
|\langle w, v \rangle_{L^2(U)} 
	-  \langle u, v \rangle_{L^2(\Xpm(U))}|
	 \lesssim \epsilon \| w \|_{L^2(\Sigma)} \| v \|_{C^1(N)},
\\
|\langle w, v \rangle_{L^2(\Omega)}|  \lesssim \epsilon \| w\|_{L^2(\Sigma)} \| v \|_{C^1(N)},
\quad
\text{and}
\\
|\langle u, v \rangle_{L^2(\Pi^{-1}_\Sigma(\Omega))}|  \lesssim \epsilon  \| w \|_{L^2(\Sigma)} \| v \|_{C^1(N)}.
\end{gathered}
\end{equation*}
After adding and subtracting $\int_{U} w \Xpm^* v d\sigma$, the difference on the left hand side of the first estimate can be written as $(I) + (II)$, where 
 \begin{equation*}
\begin{gathered}
(I) = \int_{U} w ( v - \Xpm^* v) d\sigma, 
\quad
(II) = \int_{U} w  \Xpm^* v ( d\sigma - \Xpm^* d\sigma). 
\end{gathered}
\end{equation*}
The first estimate follows by estimating $(I)$ and $(II)$ as in the proof Lemma \ref{Lw0uest1}.  The second and third estimates follow from the Cauchy-Schwarz inequality and \ref{Lw0uest2}(i) and \ref{Lw0uest2}(iii). 

For (iii), suppose $u \in \Coord^\perp$.  By taking $v$ in (ii) to be a constant or coordinate function, it follows that 
\begin{align*}
| \langle w, v \rangle_{L^2(\Sigma)} | \lesssim \epsilon \| w \|_{L^2(\Sigma)} \| v \|_{L^2(\Sigma)}. 
\end{align*}
Item (iii) follows from this estimate and the Pythagorean theorem, upon expanding $\wlow$ into an $L^2(\Sigma)$-orthonormal basis for $\Ecal_{\lambda_0}(\Sigma)\oplus \Lapone(\Sigma) = \R \oplus \Coord(\Sigma)$.
\end{proof}

We are now ready to prove the main result of this subsection, which characterizes $\Ecal^+_{\lambda_1}$. 
\begin{prop}
\label{Pee}
\label{Ceme}
Either $\lambda_1(\Mbreve) < 2$ and 
$\Lapone^+(\Mbreve) = \{ 0\}$, or
$\lambda_1(\Mbreve) = 2$ and $\Lapone^+(\Mbreve)= \Coord^+(\Mbreve)$.
\end{prop}
\begin{proof}
It clearly suffices to prove that $\Ecal^+_{\lambda_1}\cap \Coord^\perp = \{0\}$.  To this end, fix $u \in \Lapone^+$.  Using the decomposition $w=\wlow + \whigh$ from \ref{dw0} implies
\begin{align*}
\| d w \|^2_{L^2(\Sigma)}
& \geq 
\| d \whigh\|^2_{L^2(\Sigma)}
\\
&\geq
\lambda_2(\Sigma) \| \whigh \|^2_{L^2(\Sigma)}
\\
&=
\lambda_2(\Sigma) ( \| w\|^2_{L^2(\Sigma)} - \| \wlow \|^2_{L^2(\Sigma)}).
\end{align*}
If additionally $u \in \Coord^\perp$, combining the preceding with \ref{Cray0}(i) and \ref{Cray0}(iii) shows that
\begin{align}
\label{Eray2}
\lambda_1(\Sigma) (1+C\epsilon) \|w\|^2_{L^2(\Sigma)} \geq
\| d w\|^2_{L^2(\Sigma)}
 \geq \lambda_2(\Sigma) (1-C\epsilon) \| w \|^2_{L^2(\Sigma)}.
\end{align}
Since $\lambda_1(\Sigma) = 2 < 6= \lambda_2(\Sigma)$, taking $\epsilon>0$ sufficiently small implies $w = 0$, hence $u = 0$. 
\end{proof}

\subsection*{Proof of Theorem \ref{Tsph}}
\phantom{ab}

Using the preceding analysis of $\Ecal^-_{\lambda_1}$ and $\Ecal^+_{\lambda_1}$, we are ready for the proof. 
\begin{proof}[Proof of Theorem \ref{Tsph}]
On one hand, the nontriviality of $\Ecal_{\lambda_1}$ implies at least one of the factors in the decomposition $\Ecal_{\lambda_1} = \Ecal^+_{\lambda_1} \oplus \Ecal^-_{\lambda_1}$ is nontrivial.

On the other hand, by \ref{Cem} and \ref{Ceme}, the nontriviality of either $\Ecal^+_{\lambda_1}$ or $\Ecal^-_{\lambda_1}$ implies the nontriviality of the other, and further that $\Ecal^{\pm}_{\lambda_1} = \Coord^\pm$.  
Thus $\Ecal_{\lambda_1} = \Ecal^+_{\lambda_1} \oplus \Ecal^-_{\lambda_1} = \Coord^+ \oplus \Coord^- = \Coord$. 
\end{proof}

\section{The first eigenspace for minimal doublings of the Clifford torus}
\label{SYauCliff}
\subsection*{Preliminaries} 
\phantom{ab}

Let $(\Sigma, N, g) = (\Sigma, \Sph^3, g)$ be the \emph{background} (recall the notation of \cite[Convention 2.1]{LDG}) consisting of a closed surface $\Sigma$  embedded in the round three-sphere $(N, g) = (\Sph^3, g)$,  let $\Mbreve$ be a doubling of $\Sigma$ as in Theorem \cite[Theorem 5.7]{LDG}.  We keep to Convention  \ref{con:eps} regarding $\epsilon>0$.

We do not yet assume the base minimal surface $\Sigma$ is the Clifford torus $\T$, but do assume a strong form of Yau's conjecture holds on $\Sigma$, made precise as follows. 

\begin{assumption}
\label{Ayaubase}
We assume the first eigenspace $\Ecal_{\lambda_1}(\Sigma)$ on the base $\Sigma$ is the span $\Ccal(\Sigma)$ of the coordinate functions on $\Sigma$.
\end{assumption}

In this general case, the two sides of $\Sigma$ in $\Sph^3$ are no longer necessarily symmetric, so there need not exist a reflectional isometry $\rho : \Mbreve \rightarrow \Mbreve$, as was the case when $\Sigma$ was totally geodesic.  Our first step is to define a  surface $(M, g')$ whose eigenvalues are close to those of $(\Mbreve, g)$ which has an exact reflectional isometry exchanging the top and bottom halves of $M$.

\begin{definition}
\label{DM}
Let $M[ \varphi, \kappaunderbold]$ be the initial surface (recall \cite[Def. 3.17]{LDG}) corresponding to $\Mbreve$ as an \cite[Theorem 5.7]{LDG}, with corresponding singular set $L$ and elevation and tilt parameters $\kappaunderbold$.  Using \cite[Def. 3.17]{LDG} we define  $M : = M[\varphi, \zerobold]$. 
\end{definition}

Since the elevation and tilt parameters $\kappaunderbold$ corresponding to the initial surface $M=M[\varphi, \zerobold]$ vanish, it follows that $M$ is invariant under the coordinate reflection $\zz \mapsto - \zz$, where $\zz$ is the signed distance function from $\Sigma$ in a system of Fermi coordinates about $\Sigma$ (see \cite[Appendix A]{LDG}).  Denote by $\rho: M \rightarrow M$ the restriction of this involution to $M$.  On $M$, define a metric $g'$ by
\begin{align}
\label{Egp}
g' : = \frac{1}{2}( g + \rho^* g),
\end{align}
where $g$ is the metric induced by the ambient metric on $\Sph^3$. 

Note that $\rho$ is an involutive isometry of $(M, g')$ which fixes the circles $\{\zz = 0\}$, which are the waists of the catenoidal bridges.

\begin{lemma}
\label{Lemmbreve}
There is a diffeomorphism $f: M \rightarrow \Mbreve$ with the property that for any $u \in W^{1,2}(\Mbreve)$, 
\begin{align*}
\| u \|_{L^2(\Mbreve, g) } &\Sim_{1+ C\epsilon} \| f^* u \|_{L^2(M, g')},\\
\| d u \|_{L^2(\Mbreve, g) } &\Sim_{1+C\epsilon} \| d f^* u \|_{L^2(M, g')}. 
\end{align*}
\end{lemma}
\begin{proof}
Recall \cite[Theorem 5.7]{LDG} that the minimal surface $\Mbreve$ is the normal graph over the initial surface $M[\varphi, \kappaunderbold]$ of a smooth function $\upphihat$, and let  $X: M[\varphi, \kappaunder] \rightarrow \Mbreve$ denote corresponding graph parametrization.  Recall also the diffeomorphism $\Fcal_{\zetaboldunder} : M[\varphi, \zerobold] \rightarrow M[\varphi, \kappaunderbold]$ between the initial surfaces $M[\varphi, \zerobold]$ and $M[\varphi, \kappaunderbold]$ defined in \cite[Lemma 5.5]{LDG}, and define $f = X \circ \Fcal_{\zetaboldunder}$. 

Using \cite[Lemma C.12]{LDG}, in order to prove the lemma, it suffices to prove that 
\begin{align}
\label{Egpp}
\| g' - f^*g : C^0( M[\varphi, \zerobold], g) \| < C\epsilon
\end{align}
where $C> 0$ depends only on the background. 

Using the the smallness of $\upphihat$ in terms of $\tau_{\max}$ from \cite[Theorem 5.7]{LDG}, it follows that 
\[ \| g  - X^*g : C^0(M[ \varphi, \kappaunderbold], g)\|< \epsilon. \]
Similarly, arguing as in the proof of Lemma \cite[Lemma 5.5]{LDG} and \cite[Lemma 4.3]{LDG} establishes an analogous estimate on $g - \Fcal_{\zetaboldunder}^* g$ on $M[\varphi, \zerobold]$.
 
Finally, let $g'$ be defined as in \eqref{Egp}.  Letting $\gcir$ be the Euclidean metric on the system of Fermi coordinates (recall \cite[Def. 2.2]{LDG}) and using the smallness of $g-\gcir$ from \cite[Lemma 2.22]{LDG}, it follows that
\begin{align*}
\| g' - g : C^0(M[\varphi, 0], g)\| < \epsilon.
\end{align*}
Finally, combining the preceding estimates proves \eqref{Egpp} and proves the lemma.
\end{proof}

Because $(M, g')$ is a perturbation of the minimal surface $(\Mbreve, g)$, the coordinate functions $\Coord(M)$ are not necessarily first eigenfunctions on $(M, g')$.  To handle this, we instead consider a nearby space $\Coord'(M)$ spanned by eigenfunctions with eigenvalues near $2$.  
\begin{lemma}
\label{Lproj} 
There exists a subspace $\Coord'(M) \subset L^2(M)$ with the following properties.
\begin{enumerate}[label=\emph{(\roman*)}]
\item $\Coord'$ is four-dimensional and a direct sum of subspaces of eigenfunctions. 
\item $\| v - P v \|_{L^2(M)} \leq \epsilon \| v \|_{L^2(M)}$ for each $v \in \Coord$, where $P : \Coord(M) \rightarrow \Coord'(M)$ is the orthogonal projection.
\end{enumerate}
\end{lemma}
\begin{proof}
Note that $\Coord(\Mbreve) \subset \Ecal_{\lambda_1}(\Mbreve, g)$ and that $\Coord(\Mbreve)$ is four-dimensional.  The conclusion now follows from the closeness of $(\Mbreve, g)$ to $(M, g')$ from Lemma \ref{Lemmbreve}, and the continuous dependence of eigenvalues on $C^2$ perturbations of the metric \cite[Lemma 2.2]{FS1}.
\end{proof}

When the role of $(M,g')$ is clear, we write $\Ecal_{\lambda_i}$ in place of $\Ecal_{\lambda_i}(M, g')$.  Just as in Section \ref{Ssph}, the involutive isometry $\rho$ of $(M, g')$ induces even-odd decompositions on $\rho$-invariant subspaces of $L^2(M, g')$.  In the next subsections, we study even and odd eigenfunctions whose eigenvalues are close to $2$; for this reason it is convenient to make the following definition. 
\begin{definition}
\label{deappr}
Let $\Eappr : = \bigoplus_{i} \Ecal_{\lambda_i}$, where the sum is over the $i$ such that $\lambda_i \in (1/4, 2+\epsilon)$.
\end{definition}

\subsection*{The even part $\Eappr^+$}
\phantom{ab}

We will study $\Eappr^+$ in essentially the same way we studied $\Ecal^+_{\lambda_1}(\Mbreve)$ from Section \ref{Ssph}, when $\Mbreve$ was a symmetric $\Spheq$-doubling.   For clarity and completeness, we repeat the corresponding discussion from \ref{Ssph}, making modifications as appropriate to accommodate the more general setting.

\begin{definition}
\label{dOmegap1}
For each $p \in L$, define $\Omegap = D^\Sigma_p(\tau^{2\alpha}_p)$, $\Omega = \cup_{p \in L} \Omega_p$, and and $U = \Sigma \setminus \Omega$. 
\end{definition}

Recall from \cite[Def. 3.17]{LDG} that $\Pi^{-1}_\Sigma(\Omega)$ is the union of the graphs of functions $\varphi^{gl}_+$ and $-\varphi^{gl}_-$, where the functions $\varphi^{gl}_{\pm}$ are defined on $U$ and satisfy (recall \cite[Lemma 3.19]{LDG})
\begin{align}
\label{Eubreve2}
\| \varphi^{gl}_{\pm} : C^{2, \beta}(U) \| \lesssim \tau^{8/9}_{\max}.
\end{align}
Note the vanishing of the elevation and tilt parameters in the definition of $M$ implies moreover that $\varphi^{gl}_{+} = \varphi^{gl}_-$. 
We denote by $X_{\pm}: U \rightarrow \Mbreve$ the maps parametrizing the graphs of $\varphi^{gl}_{+}$ and $-\varphi^{gl}_-= - \varphi^{gl}_+$.

\begin{definition}[The approximate eigenfunction $w$]
\label{dw}
Given $u \in \Eappr^+$, 
define a piecewise smooth function $w$ on $\Sigma$ by requesting that
\begin{align}
\label{Ewdef}
w = \Xp^* u = \Xm^* u
\quad
\text{on} 
\quad
\Omega
\quad
\text{and}
\quad
\Delta w = 0
\quad
\text{on} 
\quad
\Sigma \setminus \Omega
\end{align}
and a decomposition $w = \wlow + \whigh$, where $\wlow$ is the projection of $w$ onto $\Ecal_{\lambda_0}(\Sigma)\oplus \Ecal_{\lambda_1}(\Sigma)$.
 \end{definition}
 
\noindent Note that $w$ is well-defined in \ref{dw} because the evenness of $u$ and the symmetry imply $\Xp^* u = \Xm^* u$. 

\begin{convention}
Unless otherwise stated, we adopt the convention that $L^2$ norms computed over domains of $M$ will be computed with respect to the metric $g'$. 
\end{convention}

The next two results are the analogs of \ref{Lw0uest1} and \ref{Lw0uest2} in the present context, and the proofs are omitted because they are essentially the same. 

\begin{lemma}
\label{Lwuest1}
For $u$ and $w$ as in \ref{dw}, the following hold. 
\begin{enumerate}[label=\emph{(\roman*)}]
	\item $\| w \|_{L^2( U)} \Sim_{1+C\epsilon} \| u\|_{L^2( \Xpm(U))}$. 
	\item $\| d w \|_{L^2( U)} \Sim_{1+C\epsilon} \| d u\|_{L^2(\Xpm(U))}$.
	\end{enumerate}
\end{lemma}

\begin{lemma}
\label{Lwuest2}
For $u$ and $w$ as in \ref{dw}, the following hold. 
\begin{enumerate}[label=\emph{(\roman*)}]
	\item $\| u \|_{L^2(\Pi^{-1}_\Sigma(\Omega))} \lesssim \epsilon \| u \|_{L^2(M)}$. 
	\item $\| d u \|_{L^2(\Pi^{-1}_\Sigma( \Omega))} \lesssim \epsilon \| d u \|_{L^2(M)}$. 
	\item $\| w \|_{L^2(\Omega)} \lesssim \epsilon \| w \|_{L^2(\Sigma)}$. 
	\item $\| d w  \|_{L^2(\Omega)} 
	 \lesssim \epsilon \| d w \|_{L^2(\Sigma)}$. 
	\end{enumerate}
\end{lemma}

\begin{cor}
\label{Cray1}
For $u$ and $w$ as in \ref{dw}, the following hold.
\begin{enumerate}[label=\emph{(\roman*)}]
\item $\| \nabla w \|^2_{L^2(\Sigma)}
\leq 
(1+ C\epsilon) \lambda_1(\Sigma) \| w \|^2_{L^2(\Sigma)}.$
\item $|2 \langle w , v \rangle_{L^2(\Sigma)} 
	-  \langle u, Pv \rangle_{L^2(M)}|
	 \lesssim \epsilon \| w \|_{L^2(\Sigma)}  \| v \|_{C^1(N)}$ for any  $v \in \Coord(M)$.
\item $ \| \wlow\|_{L^2(\Sigma)}  \lesssim \epsilon\| w \|_{L^2(\Sigma)}$, if $u \in \Coord'^\perp$.
\end{enumerate}
\end{cor}
\begin{proof}
The proof of (i) is omitted because it is essentially the same as the proof of \ref{Cray0}(i).  For (ii), arguing as in the proof of \ref{Cray0}(ii) shows that 
\begin{align*}
|2 \langle w , v \rangle_{L^2(\Sigma)} 
	-  \langle u, v \rangle_{L^2(M)}|
	 \lesssim \epsilon \| w \|_{L^2(\Sigma)}  \| v \|_{C^1(N)}
\end{align*}
for any coordinate function $v$.  On the other hand, we have also
\begin{align*}
|\langle u, v \rangle_{L^2(M)} - \langle u, Pv \rangle_{L^2(M)} | \leq \| u \|_{L^2(M)} \| v - P v \|_{L^2(M)},
\end{align*}
and (ii) follows from combining the preceding with \ref{Lproj}.

For (iii), suppose $u \in \Coord'^\perp$.  By taking $v$ in (ii) to be a coordinate function, it follows that 
\begin{align*}
| \langle w, v \rangle_{L^2(\Sigma)} | \lesssim \epsilon \| w \|_{L^2(\Sigma)} \| v \|_{L^2(\Sigma)}. 
\end{align*}
Combining this with the analogous and easier estimate where $v$ is a constant function, (iii) follows from this estimate and the Pythagorean theorem, upon expanding $\wlow$ into an $L^2(\Sigma)$-orthonormal basis for $\Ecal_{\lambda_0}(\Sigma)\oplus \Lapone(\Sigma) = \R \oplus \Coord(\Sigma)$.
\end{proof}

\begin{lemma}
\label{Lecj}
$\Eappr^+ \subset \Coord'$.
\end{lemma}
\begin{proof}

The proof is very similar to the proof of \ref{Pee}, and it clearly suffices to prove that $\Eappr^+\cap \Coord'^\perp = \{0\}$.  To this end, fix $u \in \Eappr^+$, and let $w$ be as in \ref{dw}; 
arguing as in the proof of \ref{Pee} shows that
\begin{align*}
\| d w \|^2_{L^2(\Sigma)}
&\geq
\lambda_2(\Sigma) ( \| w\|^2_{L^2(\Sigma)} - \| \wlow \|^2_{L^2(\Sigma)}).
\end{align*} 
If additionally $u \in \Coord'^\perp$, combining the preceding with \ref{Cray1}(i) and \ref{Cray1}(iii) shows that
\begin{align}
\label{Eray}
2 (1+C\epsilon) \|w\|^2_{L^2(\Sigma)} \geq
\| d w\|^2_{L^2(\Sigma)}
 \geq \lambda_2(\Sigma) (1-C\epsilon) \| w \|^2_{L^2(\Sigma)}.
\end{align}
Since $2< \lambda_2(\Sigma)$, taking $\epsilon>0$ sufficiently small implies $w=0$, hence $u =0$.
\end{proof}

\subsection*{The odd part $\Eappr^-$}
\phantom{ab}

\label{S:nontot}
In order to characterize $\Eappr^-$, we now assume the base surface $\Sigma$ is a Clifford torus doubling:
\begin{assumption}
\label{Abd}
We assume $\Mbreve \subset \Sph^3$ is a doubling of the Clifford torus constructed in Section \ref{Scliff}.
\end{assumption}

\begin{convention}
\label{Ci}
In this subsection, denote by $i \in \N$ the smallest positive integer such that $\Ecal^-_{\lambda_i}(M, g')$ is nontrivial.
\end{convention}

In Proposition \ref{Peigen}, we show that $\Eappr^- = \{0\}$ by estimating $\lambda_i= \lambda_i (M, g')$, and showing it is close to $4$.  To begin, let $u\in \Ecal^-_{\lambda_i}(M)$ be nonzero.  Since $u$ vanishes on $\{\zz= 0\}$, which separates $M$, it follows from the Courant nodal domain theorem that $u$ spans $\Ecal^-_{\lambda_i}(M, g')$.

\begin{lemma}
\label{Lgroup2}
$\Ecal^-_{\lambda_i}(M, g')$ is $\group$-invariant, that is $\gbold^*v = v$ for each $v \in \Ecal^-_{\lambda_i}(M)$ and  each $\gbold \in \group$. 
\end{lemma}
\begin{proof}
Because $u$ spans $\Ecal^-_{\lambda_i}(M)$, it suffices to prove the lemma for $v = u$.  Note that $\gbold^* u = u \circ \gbold \in \Ecal^-_{\lambda_i}(M)$, hence must be a multiple of $u$.  Since $\gbold$ is an isometry of $(M, g')$, this multiple must be $1$, and the proof is complete.
\end{proof}

\begin{definition}
\label{dweo}
Given $u \in \Ecal^-_{\lambda_i}(M, g')$, define a piecewise smooth function $w$ on $\Sigma$ by requesting that
\begin{align*}
w &= \Xp^* u =- \Xm^* u 
\quad
\text{on}
\quad
\Omega
\quad
\text{and}
\quad
\Delta w = 0
\quad
\text{on} 
\quad
\Sigma \setminus \Omega.
\end{align*}
Define also the decomposition $w = \wave + \wosc$, where $\wave$ is rotationally symmetric as in \ref{dave}.
\end{definition}

\begin{lemma}
\label{Luo}
Given $u$ and $w$ as in \ref{dweo}, the following hold. 
\begin{enumerate}[label=\emph{(\roman*)}]
\item $\| u\|^2_{L^2(M)} \Sim_{1+C\epsilon} 2\| w \|^2_{L^2(\Sigma)}$.
\item $\| d w\|_{L^2(U)} \Sim_{1+C\epsilon} \|d u \|_{L^2(X_\pm(U))}.$
\end{enumerate}
\end{lemma}
\begin{proof}
We omit the proof because it is similar to the proof of \ref{Lwuest1}.
\end{proof}

\begin{prop}
\label{Peigen}
Let $\Mbreve$ be as in \ref{Abd} and $\lambda_i$ be as in \ref{Ci}.  Then $\lambda_i \geq 4(1-C \epsilon)$.  In particular, $\Eappr^-= \{ 0\}$. 
\end{prop}
\begin{proof}
Using that $M = \Pi^{-1}_\Sigma ( U) \cup \Pi^{-1}_\Sigma(\Omega)$ and \ref{Luo}, we estimate
\begin{align}
\label{Eupp0}
\|d u \|^2_{L^2(M)} \geq  \| d u\|^2_{L^2(\Pi^{-1}_\Sigma(\Omega))} 
+ 2(1-C\epsilon) \| d \wosc\|^2_{L^2(\Sigma)}.  
\end{align}
By Lemma \ref{Lgroup2}, $w$ is $\group$-invariant; by arguing as in the proof of \ref{LPhipest} and estimating a uniform lower bound for the first eigenvalue of $\Delta$ on $\Sigma$, when restricted to $\group$-symmetric functions with average zero, it follows that 
\begin{align*}
\| d \wosc \|^2_{L^2(\Sigma)} \geq Cm^2 \| \wosc \|^2_{L^2(\Sigma)}, 
\end{align*}
where $m \in \N$ may be taken as large as needed.  In combination with \eqref{Eupp0}, this implies
\begin{align}
\label{Eupp}
\| d u \|^2_{L^2(M)} \geq  \| d u \|^2_{L^2(\Pi^{-1}_\Sigma(\Omega))}  + 2(1-C \epsilon) Cm^2\| \wosc \|^2_{L^2(M)}.
\end{align}
If $\| \wosc\|^2_{L^2(\Sigma)} > \epsilon \| w \|^2_{L^2(\Sigma)}$, then in combination with \eqref{Eupp} and Lemma \ref{Luo}(ii), we find
\begin{align*}
\| d u \|^2_{L^2(M)} \geq 2(1-C\epsilon) \epsilon m^2 \| u \|^2_{L^2(M)}, 
\end{align*} 
and the conclusion of the Proposition follows by taking $m$ large enough in terms of $\epsilon$. 

We may therefore assume that
\begin{align}
\label{Ewsm}
\| \wosc\|^2_{L^2(\Sigma)} \leq \epsilon \| w \|^2_{L^2(\Sigma)}, 
\quad 
\text{hence}
\quad
\| w \|^2_{L^2(\Sigma)} \Sim_{1+C\epsilon} \|\wave\|^2_{L^2(\Sigma)}. 
\end{align}

Now let $\Omega' : = D^{\Sigma}_L(m/ 100)$.  Because $\Delta u + \lambda_i u = 0$, using the closeness of the metrics $g'$ and $X^* g'$,  \eqref{Ewsm} and standard elliptic theory, it follows that
\begin{align}
\label{Eweq}
\Delta w + \lambda w = E^{err}
\quad
\text{on}
\quad
\Sigma \setminus \Omega',
\quad
\text{where}
\quad
\| E^{err} \|_{C^{2, \beta}(\Sigma\setminus \Omega')} \leq C \epsilon \| w\|_{L^2(\Sigma)}. 
\end{align}
For convenience, we now scale $u$ to assume without loss of generality that $\|w\|_{L^2(\Sigma)} = 1$.

\emph{Case 1: $\Mbreve$ is as in Theorem \ref{Tcliff}.}  In this case, each fundamental domain for $\group$ on $M$ contains one catenoidal bridge. 
By Lemma \ref{Lgammad}, the function $\sss : = \dbold_{\Lpar}$ is smooth for $\sss \in (0, \sss_{\max} )$, where $\sss_{\max}: = \pi/ (\sqrt{2} k |v|)$.   It follows from \ref{Ewsm} that for $\sss \in (m/100, \sss_{\max} )$, $\wave$ satisfies the ODE
\begin{align}
\label{Ewode}
\partial^2_{\sss} \wave + \lambda \wave = E^{err}_{\ave}.
\end{align}
Clearly $w_{\ave}$ is positive for $\sss \in (m/100, \sss_{\max})$, and a short calculation using \eqref{Ewode} shows the function $F^{\wave}$ defined by $F^{\wave}(\sss) : = \frac{\partial_{\sss} \wave}{\wave}$ satisfies 
\begin{align}
\label{Ewflux}
\frac{d F^{\wave}}{d \sss} = \lambda + \left( F^{\wave}\right)^2 + O(\epsilon). 
\end{align}
Now let $\Omegapar' : = D_{\Lpar}(m/100)$.  
By integrating \eqref{Eweq} on $\Omegapar' \setminus \Omega'$ and integrating by parts, we find 
\begin{align}
\label{Ewflux11}
- \int_{\partial \Omega'} \frac{\partial w}{\partial \nu} + \int_{\partial \Omegapar'} \frac{\partial w}{\partial \nu} +  \lambda \int_{\Omegapar' \setminus \Omega'} w = \int_{\Omegapar' \setminus \Omega'} E^{err}_{\ave}.
\end{align}
Dividing through in \eqref{Ewflux11} by $\wave|_{\{\sss = m/100\}}$, estimating the leftmost integral in \eqref{Ewflux11} using the logarithmic behavior of $w$, and using \eqref{Ewsm} and \eqref{Eweq}, we conclude
\begin{align}
\label{Ewflux2}
\sum_{p \in L} \frac{2\pi}{|\log m \tau_p|} = \int_{\partial \Omegapar'} F^{\wave}+ O(1/m) + O(\epsilon).
\end{align}

On the other hand, now let $\varphi$ be the LD solution associated to $\Mbreve$.  Its average $\varphi_{\ave}$ satisfies
\begin{align}
\partial^2_{\sss} \varphi_{\ave} + 4 \varphi_{\ave} = 0
\quad
\text{for} 
\quad
\sss \in (m/100, \sss_{\max})
\end{align}
and the function $F^{\varphi_{\ave}}$ defined by $F^{\varphi_{\ave}} : = \frac{\partial_{\sss} \varphi_{\ave}}{\varphi_{\ave}}$ satisfies
\begin{align}
\label{Efflux}
\frac{d F^{\varphi_{\ave}}}{d \sss} = 4 + \left( F^{\varphi_{\ave}}\right)^2.
\end{align}
Furthermore, integrating the equation $\Delta \varphi + 4 \varphi = 0$ on $\Omegapar' \setminus \Omega'$, and integrating by parts shows that
\begin{align}
\label{Efflux2}
\sum_{p \in L} \frac{2\pi}{|\log m\tau_p|} = \int_{\partial \Omegapar'} F^{\varphi_{\ave}}+O(1/m).
\end{align}
Finally, comparing \eqref{Ewflux2} with \eqref{Efflux2} shows that $F^{\wave} = F^{\varphi_{\ave}} + O(1/m) + O(\epsilon)$ along $\partial \Omegapar'$.  By the symmetry, note that
\begin{align*}
0 = F^{\wave} = F^{\varphi_{\ave}}
\quad
\text{along}
\quad
\sss = \sss_{\max}.
\end{align*}
It follows from the preceding that $4 - \lambda_i =O(\epsilon)$
completing the proof of the Proposition in this case.
\emph{Case 2: $\Mbreve$ is as in Theorem \ref{Tcliff2}}.  In this case, each fundamental domain for $\group$ on $M$ contains three catenoidal bridges.  The structure of the proof in this case is extremely similar to the argument in Case 1, so we omit the details. 
\end{proof}

\subsection*{The main theorem}
\phantom{ab}

\begin{theorem}
\label{TcliffE}
Let $\Mbreve$ be a minimal doubling of the Clifford torus constructed in Section \ref{Scliff}.  Then the first eigenspace of the Laplacian on $\Mbreve$ is spanned by the coordinate functions.  In particular, $\lambda_1(\Mbreve) =2$.
\end{theorem}
\begin{proof}
Since the first eigenspace $\Ecal_{\lambda_1}(\Mbreve, g)$ is nontrivial and since $\Coord(\Mbreve) \subset \Ecal_2(\Mbreve, g)$,  it suffices to prove that $\Ecal_{\lambda_1}(\Mbreve, g) \cap \Coord^\perp(\Mbreve) = \{0\}$.

Let $u \in \Ecal_{\lambda_1}(\Mbreve, g)$, and let $u'$ be the projection of $u$ onto $\Eappr$.   By \ref{Lecj} and \ref{Peigen}, we have $\Eappr = \Coord'$.  Further, by \ref{Lproj} there is a coordinate function $v \in \Coord(\Mbreve)$ such that $Pv = u'$, where $P : \Coord \rightarrow \Coord' = \Eappr$ is the projection.  By the Cauchy-Schwarz inequality, we have
\begin{align*}
\langle u, v \rangle_{L^2(\Mbreve)} &= \| u\|^2_{L^2(\Mbreve)} + \langle u, u-v \rangle_{L^2(\Mbreve)} 
\\
&\geq 
\| u \|^2_{L^2(\Mbreve)} - \| u \|_{L^2(\Mbreve)} \| u - v \|_{L^2(\Mbreve)}.
\end{align*}
On the other hand, by \ref{Lemmbreve}, \ref{Lproj}(ii), and the triangle inequality, we have
\begin{align*}
\| u - v\|_{L^2(\Mbreve)} &\leq \| u - u'\|_{L^2(\Mbreve)} + \| u' - v \|_{L^2(\Mbreve)}\\
&\leq C\epsilon \| u \|_{L^2(\Mbreve)}.
\end{align*}
Consequently, $\langle u, v \rangle_{L^2(\Mbreve)} \geq (1-C\epsilon) \| u \|^2_{L^2(\Mbreve)}$.  By taking $\epsilon>0$ small enough, we see that if $u \neq 0$, then $\langle u, v \rangle_{L^2(\Mbreve)} > 0$, which completes the proof.
\end{proof}


\appendices
\section*{Appendices}

\section{Eigenvalues of the Laplacian on Flat Tori}
\label{Stori}
$\phantom{ab}$
\nopagebreak

This appendix catalogues some well-known results \cite{Berger, MontielRos} concerning the spectrum of the Laplacian on flat, $2$-dimensional tori, following the presentation in \cite{MontielRos}. Up to dilations, each such torus is well-known to be isometric to a quotient $\R^2 / \Gamma$, where $\Gamma = \Gamma[a, b]$ is a lattice generated by $\{ (1, 0), (a, b)\}$ with
\begin{align*}
0 \leq a \leq \frac{1}{2}, 
\quad
b > 0,
\quad
a^2 + b^2 \geq 1. 
\end{align*}

The eigenvalues of the Laplacian of $\R^2 / \Gamma$ with the induced Euclidean metric are given by \cite[p. 146]{Berger}
\begin{align*}
\lambda_{pq} = 4\pi^2 \bigg( q^2 + \left( \frac{p - q a }{b}\right)^2 \bigg), 
\end{align*}
where $(p, q ) \in S = \{ ( r, s) \in \Z \times \Z : s \geq 0 \text{ or } s = 0 \text{ and } r \geq 0\}.$

Furthermore, the eigenspace corresponding to $\lambda_{pq}$ is spanned by 
\begin{align*}
f_{pq}(x,y) = \cos 2\pi \left( q x + \frac{p-qa}{b} y \right), \\
g_{pq}(x,y) = \sin 2\pi \left( qx + \frac{p-qa}{b} y \right). 
\end{align*}

\bibliographystyle{abbrv}
\bibliography{bibliography}
\end{document}